\documentclass[12pt,reqno]{amsart}
\usepackage{hyperref}
\usepackage{amsmath,amssymb}
\usepackage[english]{babel}
\usepackage{caption}
\usepackage{amssymb,amsmath,euscript,enumerate,tikz}
\usepackage{soul}
\usepackage[margin=1in]{geometry}
\usepackage{graphicx}
\usepackage{float}
\usepackage{pgf,tikz}
\usepackage{mathrsfs}
\usepackage{upgreek}
\newcommand{\sk}{{\ensuremath{\sf k }}}

\DeclareMathOperator{\dist}{dist}

\newtheorem{conjecture}{Conjecture}[section]
\newtheorem{theorem}[conjecture]{Theorem}
\newtheorem{lemma}[conjecture]{Lemma}
\newtheorem{corollary}[conjecture]{Corollary}

\theoremstyle{definition}
\newtheorem{remark}[conjecture]{Remark}

\newtheorem{definition}[conjecture]{Definition}
\newtheorem{notation}[conjecture]{Notation}

\renewcommand\dim{\text{\rm dim}}

\providecommand\hgt{\text{\rm ht}}

\providecommand\max{\text{\rm max}}

\providecommand\reg{{\rm reg}}

\begin{document}
\title{Powers of facet ideals of simplicial trees}
\author[Ajay Kumar]{Ajay Kumar}
\email{ajay.kumar@iitjammu.ac.in}
\address{Department of Mathematics, Indian Institute of Technology Jammu, J\&K, India - 181221.}
\author[Arvind Kumar]{Arvind Kumar}
\email{arvindkumar@cmi.ac.in}
\address{Department of Mathematics, Chennai Mathematical Institute, Siruseri
Kelambakkam, India -603103.}

\begin{abstract}
In this article, we study the linearity of the minimal free resolution of powers of facets ideals of simplicial trees. We give a complete characterization of simplicial trees for which (some) power of its facet ideal  has a linear resolution. We calculate the regularity of the $t$-path ideal of a perfect rooted tree. We also obtain  an upper bound for the regularity  of the $t$-path ideal of a rooted tree. We give a procedure to calculate the regularity of powers of facet ideals of simplicial trees. As a consequence of this result, we study the regularity of powers of $t$-path ideals of rooted trees. We pose a regularity upper bound conjecture for facet ideals of simplicial trees, which is as follows: if $\Delta$ is a $d$-dimensional simplicial tree, then $\reg(I(\Delta)^s) \leq (d+1)(s-1)+\reg(I(\Delta))$ for all $s \geq 1$.  We prove this conjecture for some special classes of simplicial trees.
\end{abstract}

\maketitle
\section{Introduction}
Studying homological properties of powers of square-free monomial ideals is an active research topic in commutative algebra.  There is a one-to-one correspondence between square-free monomial ideals and simplicial complexes. To each simplicial complex $\Delta$ on the vertex set $\{x_1,\ldots,x_n\},$ one can associate a square-free monomial ideal 
$$I(\Delta)=\left( \left\{ \prod\limits_{x \in F}x: \text{F~ is~ a ~facet ~of}~ \Delta \right\}\right).$$ In particular, any graph $G$ can be naturally viewed as a $1$-dimensional simplicial complex, and the facet ideal $G$ is known as the edge ideal in the literature. One of the basic problems is to find a characterization of square-free monomial ideals whose all powers have a linear resolution in terms of the combinatorial data of underlying simplicial complexes. Fr{$\ddot{\text{o}}$}berg in \cite{froberg} characterized  quadratic square-free monomial ideals having a linear resolution. In \cite{FC}, Faridi and Cannon gave a combinatorial characterization of  square-free monomial ideals with linear resolution in characteristic $2$.  Herzog and Takayama \cite{Herzog'sBook,HTY} introduced the notion of polymatroidal ideals and proved that all powers have a linear resolution for these ideals. Herzog, Hibi, and Zheng \cite{HHZ} proved that if $I \subset \sk[x_1,\ldots,x_n]$ is a monomial ideal with $2$-linear resolution, then each power has a linear resolution.  In general, if a homogeneous ideal  $I$ has a linear resolution, then all powers of $I$ need not have a linear resolution. Over a  field of characteristic  $0$, the ideal $I=(abd,abf,ace,adc,aef,bde,bcf,bce,cdf,def)$  has a linear resolution, while $I^2$ has no linear resolution. This example is due to Terai, which appeared in \cite[Remark 3]{Conca2000}. Sturmfels in \cite{Sturmfels} gave a characteristic free example of a homogeneous ideal $I$ with a linear resolution, but $I^2$ does not have a linear resolution.  Sturmfels showed that the square-free monomial ideal $I = (def, cef, cdf, cde, bef, bcd, acf, ade)$ has a linear quotient (with respect to the given order on generators), but $I^2$ has no linear resolution. In particular, $I^2$ does not have a linear quotient with respect to any ordering of minimal generators.   

This article focuses on the powers of facet ideals of  simplicial trees. Zheng \cite{Zheng} gave a combinatorial characterization of simplicial trees 
 whose facet ideals admit a linear resolution. We study the linearity of the minimal free resolution of all powers of facet ideals of simplicial trees. The first main result of this article is as follows: \\\\
{\bf Theorem A.} (see Theorem \ref{linear})
Let $\Delta$ be a simplicial tree on the vertex set $[n].$ Then, the followings are equivalent:
\begin{enumerate}
\item $\Delta$ satisfies intersection property.
    \item $I(\Delta)$ has a linear resolution.
    \item $I(\Delta)^s$ has linear quotients for all $s \ge 1$ and $\Delta$ is pure.
    \item $I(\Delta)^s$ has linear quotients for some $s$ and $\Delta $ is pure.
    \item $I(\Delta)^s$ has a linear resolution for some $s.$
    \item $I(\Delta)^s$ has linear first syzygies for some $s.$
\end{enumerate}

The Castelnuovo–Mumford regularity (in short, regularity) of a homogeneous ideal is one of the most important invariants in commutative algebra.  Section \ref{section4} studies the regularity of $t$-path ideals of a rooted tree $\Gamma$.  The $t$-path ideal of a directed graph $\Gamma$, denoted by $I_t(\Gamma)$, was introduced by He and Van Tuyl in \cite{VanH}. They proved that the $t$-path ideal of a rooted tree is the facet ideal of some simplicial tree, and the Betti numbers, in this case, do not depend on the characteristic of the base field. Bouchat, H\`a, and O'Keefe in \cite{BRT} studied the $t$-path ideals of a rooted tree and gave a recursive formula to compute the graded Betti numbers of $I_t(\Gamma)$ and also obtained a general bound for the regularity of $I_t(\Gamma)$. Kiani and Madani in \cite{KMS} calculated the  graded Betti numbers of the path ideal of a  rooted tree in terms of combinatorial invariants of the associated rooted tree.  In this article, we obtained the regularity of $t$-path ideal of a perfect rooted tree (see Section \ref{section4} for the definition) when $\left\lceil{\frac{\hgt(\Gamma)+1}{2}}\right\rceil \leq t \leq \hgt(\Gamma) +1$, where $\hgt(\Gamma)$ denotes the height of $\Gamma$. \\\\
{\bf Theorem B.} (see Theorem \ref{maint}) Let $(\Gamma,x_0)$ be a perfect rooted tree with $\hgt(\Gamma) \geq 1$ and let $t$ be a positive integer such that $\left\lceil{\frac{\hgt(\Gamma)+1}{2}}\right\rceil \leq t \leq \hgt(\Gamma) +1$.  Then $$ \reg\left(\frac{R}{I_t(\Gamma)}  \right) = \sum\limits_{\ell_{\Gamma}(x)=\hgt(\Gamma) -t}^{\hgt(\Gamma)-2}\deg^{+}_{\Gamma}(x).$$

As a consequence of the above result, we obtain an improved upper bound for the regularity of $I_t(\Gamma)$ when  $\left\lceil{\frac{\hgt(\Gamma)+1}{2}}\right\rceil \leq t \leq \hgt(\Gamma) +1$.\\\\
{\bf Theorem C} (see Corollary \ref{cormain}).
Let $(\Gamma,x_0)$ be a rooted tree with $\hgt(\Gamma)\geq 1$ and let $t$ be a positive integer such that $\left\lceil{\frac{\hgt(\Gamma)+1}{2}}\right\rceil \leq t \leq \hgt(\Gamma) +1$. Then 
$$\reg\left(\frac{R}{I_t(\Gamma)}\right) \leq  \sum\limits_{\ell_{\Gamma}(x)=\hgt(\Gamma) -t}^{\hgt(\Gamma)-2}\deg^{+}_{\Gamma}(x)+\sum\limits_{i=t-1}^{\hgt(\Gamma)-2}(\hgt(\Gamma)-i-1)l_{i}(\Gamma),$$ where $l_{i}(\Gamma)$ denotes the number of leaves in $\Gamma$ at the level $i$.

Studying the regularity of powers of homogeneous ideals has been a central research topic in commutative algebra and algebraic geometry in the last two decades. Cutkosky, Herzog, Trung  \cite{CHT} and independently Kodiyalam  \cite{vijay} proved that for a homogeneous ideal $I$ in a polynomial ring, the regularity function $\reg({I}^s)$ is asymptotically a linear function. Since then, many researchers have explored the regularity function for various classes of homogeneous ideals (cf. \cite{BC17,BCV15,Conca2000,HTT,HHZ,nguyen_vu,JKS20,Sturmfels}). However, more focus is given to the regularity function of edge ideals (cf. \cite{ABS,banerjee,BBSHA,BCDMS,BHT,RYJN,Er,Er2,Gu,jayanthan,JSS,JS,KK23,KKS21}). In the case of edge ideals, the regularity function can be explicitly described or bounded in terms of combinatorial invariants of associated graphs.  Banerjee, Beyarslan and H\`a \cite{BBH17} conjectured that   $\reg(I(G)^s) \leq 2s + \reg(I(G)) -2$ for any graph $G$ and for any $s \geq 1$. This conjecture has gotten much attention and has been studied by various well-known mathematicians working in commutative algebra. So far, this conjecture has been proved for various classes of graphs, namely, gap-free and cricket-free graphs, cycles and unicyclic graphs, very well-covered graphs, bipartite graphs, gap-free and diamond-free graphs, gap-free and $C_4$-free graphs, Cameron-Walker graphs, subclasses of bicyclic graphs (cf. \cite{ABS,banerjee,BBSHA,BHT,RYJN,Er,Er2,Gu,JSS,KKS21}).  However, the example  of a monomial ideal $I = (def, cef, cdf, cde, bef, bcd, acf, ade)$ due to Sturmfels \cite{Sturmfels} shows that this conjecture can not be extended to facet ideals of higher dimensional simplicial complexes. When $\Delta$ is a $d$-dimensional simplicial tree, it follows from  \cite{HHTZ} and \cite[Corollary 2.4]{hien2021} that $\reg(I(\Delta)^s) \leq (d+1)(s-1)+\dim(R/I(\Delta))+1$ for all $s$.  Based on computational evidence  and the result mentioned above, we believe that  $\reg(I(\Delta)^s) \leq (d+1)(s-1)+\reg(I(\Delta))$ for all $s$, when $\Delta$ is a $d$-dimensional simplicial tree. Therefore, we pose the following conjecture:\\\\
{\bf Conjecture D.} Let $\Delta$ be a $d$-dimensional simplicial tree on the vertex set $[n].$ Then, $$\reg(I(\Delta)^s) \leq (d+1)(s-1)+\reg(I(\Delta)) \text{ for all } s \geq 1.$$

In Section \ref{section5}, we prove this conjecture for some special classes of simplicial trees.\\\\
{\bf Theorem E.} (see Theorem \ref{broommain})
Let $(\Gamma, x_0)$ be a broom graph of $\hgt(\Gamma)=h$ and $2 \leq t \leq h+1$. Then for all $s \geq 1,$ $$\reg{\left(\frac{R}{I_t(\Gamma)^s}\right)} = t(s-1)+\reg\left(\frac{R}{I_t(\Gamma)}\right).$$ \\
{\bf Theorem F.} (see Theorem \ref{perfectr})
Let $(\Gamma,x_0)$ be a perfect rooted tree with $\hgt(\Gamma) \geq 1$ and let $t=\hgt(\Gamma)+1$.  Then,  for all $s \geq 1$, $$\reg{\left(\frac{R}{I_t(\Gamma)^s}\right)} = t(s-1)+\reg\left(\frac{R}{I_t(\Gamma)}\right).$$

\vskip 2mm
\noindent
\textbf{Acknowledgement:}  
The first author acknowledges the financial support from the Sciences and Engineering Research Board, India, under the MATRICS Scheme (MTR/2022/000543). The second author is partially supported by National Postdoctoral Fellowship (PDF/2020/001436) by the Sciences and Engineering Research Board, India. Some of this work was completed when the second author visited IIT Jammu for  ten days in Feb 2023.

\section{Regularity of powers of facet ideals of simplicial trees}\label{section3}
In this section, we study the regularity of powers of facet ideals of simplicial trees. In \cite{Zheng}, Zheng characterized simplicial trees whose facet ideals have a linear resolution. We characterize the linearity of the resolution of powers of facet ideals of simplicial trees. First, we recall some basic notation and terminology used in the paper. 

For any positive integers $l,m$, we use  notation $[l,m] $ for the set $\{k \in \mathbb{N}: l\leq k \leq m\}$. For $n \in \mathbb{N}$, we denote the set $[1,n]$ by $[n]$. Let $R=\sk[x_1,\ldots,x_n]$ be a standard graded polynomial ring over a field $\sk$ in $n$ variables and $I$ be a homogeneous ideal of $R$. The minimal graded free resolution of $I$ is given by
\[
0\rightarrow\bigoplus_jR(-j)^{\beta_{\ell,j}^R}\rightarrow\cdots\rightarrow\bigoplus_jR(-j)^{\beta_{i,j}^R}\rightarrow\cdots \rightarrow \bigoplus_jR(-j)^{\beta_{0,j}^R}\rightarrow I\rightarrow 0,
\] 
where $R(-j)$ is a graded free $R$-module of rank one generated by a homogeneous element of degree $j$. The numbers $\beta_{i,j}^R$ are called the \textit{graded Betti numbers} of  $I$. If $I$ is equigenerated by homogeneous elements of degree $t$ and $\beta_{1,j}^R(I)=0$ for all $j \neq 1+t$, then we say that $I$ has {\it linear first syzygy.} We say that $I$ admits a \textit{linear resolution} if for all $i\ge 0$, $\beta_{i,j}^R(I)=0$ for $j \neq i+t$.
The {\it regularity} of  $I$, denoted as $\text{reg}(I)$, is defined as $$\text{reg}(I)=\max\{j-i:\beta_{i,j}^R(I)\neq 0\}.$$ Observe that  $\reg(R/I) =\reg(I)-1.$

Next, we recall the concept of simplicial complex and some relevant properties which we use throughout this article. 

A simplicial complex $\Delta$ on the vertex set $[n]$ is a collection of subsets of $[n]$ which satisfies the following properties:
\begin{enumerate}[(1)]
    \item $\{i\} \in \Delta$ for all $1 \leq i \leq n$,
    \item $G \in \Delta$ whenever $F \in \Delta $ and $G \subset F.$
\end{enumerate}

Let $\Delta$ be a simplicial complex on the vertex set $[n].$ An elements $F$ of $\Delta$ is called a  \textit{face}. The \textit{dimension of a face} $F$ is $|F| - 1$. A \textit{facet} of $\Delta$ is a face that is maximal under inclusion. The  \textit{dimension} of $\Delta$, denoted by $\dim(\Delta)$, is the maximum among dimension of facets of $\Delta$. A simplicial complex is said to be \textit{pure} if all of its facets have the same dimension. 

 We denote the simplicial complex $\Delta$ with facets $F_1,\ldots,F_r$  by $\Delta=\langle F_1,\ldots,F_r  \rangle.$
 Set $m_i=\prod\limits_{j \in F_i}x_j$ for $1 \leq i \leq r.$ The {\it facet ideal of $\Delta$} is the monomial ideal in $R=\sk[x_1,\ldots, x_n]$ generated by $\{m_1 ,\ldots,m_r\}$, and it is denoted by $I(\Delta)$. A vertex $v$ is said to be {\it free vertex} if $v$ belongs to exactly one facet. A facet $F$ of $\Delta$ is said to be a {\it leaf} if either $F$ is the only facet of $\Delta$ or there exists a facet $G \neq F$ of $\Delta$ such that $H \cap F \subset G \cap F$ for all facets $H$ of $\Delta$ other than $F$. Further, a face $G$ with the above property is called a {\it branch} of $F$. A simplicial complex $\Delta$ is said to be a {\it simplicial forest} if each subcomplex of $\Delta$ has a leaf. Faridi introduced this notation in \cite{SF}. Later, Zheng in \cite{XZ} introduced the notation of a good leaf. A leaf $F$ of $\Delta$ is said to be a {\it good leaf} if $F$ is a leaf of every subcomplex of $\Delta$ containing $F$. Equivalently,   $F$ is a good leaf if $\{F \cap G \; : \;  G \text{ is a facet of } \Delta \}$ is a total order set with respect to inclusion, see \cite{HHTZ}. A {\it good leaf order on the facets of $\Delta$} is an ordering of facets  $F_1,\ldots, F_r$  such that $F_i$ is a good leaf of the subcomplex with the facets $F_1,\ldots,F_i$ for every $ 2 \leq i \leq r$. Herzog et al. in \cite{HHTZ} proved that a simplicial complex $\Delta$ is a simplicial forest if and only if the facets of $\Delta$ admit a good leaf ordering.

Next, we recall some necessary definitions from \cite{Zheng}. A simplicial complex $\Delta$ is said to be {\it connected}, if for any two facets $G$ and $H$ there is a sequence of facets
$G=G_0,\ldots,G_l=H$, satisfying  $G_j \cap G_{j+1}\neq \emptyset $ for all $j=0,1,\ldots,l.$ A sequence of facets with above property is called a {\it chain} between $G$ and $H$, and the number  $l$ is called the {\it length} of this chain.
And a simplicial complex $\Delta$ is said to be {\it connected in codimension $1$,} if for any two facets $G$ and $H$ with $\dim(G)\geq \dim (H)$, there exists a chain $\mathcal{C} : G = G_0,\ldots, G_l = H$ satisfying $\dim(G_i \cap G_{i+1})=\dim(G_{i+1}) -1$ for all $i=0,1,\ldots,l-1$, and  such a chain between $G$ and $H$ is called  a {\it proper chain } between $G$ and $H$. Let $G$ and $H$ be two facets  of a  simplicial complex $\Delta$ and  $\mathcal{C}$ be a (proper) chain between $G$ and $H$. Then $\mathcal{C}$ is {\it irredundant} if  no subsequence of  $\mathcal{C}$ except itself is a (proper) chain between $G$ and $H$.

Let $\Delta$ be a $d$-dimensional pure simplicial tree connected in codimension 1. Zheng proved that 
\cite[Propositin 1.17]{Zheng} for
any two facets $G$ and $H,$ there exists a unique irredundant proper chain between $G$
to $H$. The length of the unique irredundant proper chain between any two facets $G$ and $H$ is called the distance between $G$ and $H$, and it is denoted by $\dist_{\Delta}(G,H).$ If for any two facets $G$ and $H$, $\dim (G \cap H) = d-\dist_{\Delta}(G,H),$ then we say that $\Delta$ satisfies the {\it intersection property}.

We now study the linearity of resolution of facet ideal of simplicial trees. We first prove some auxiliary lemmas.

\begin{lemma}\label{lemmag}
Let $\Delta$ be a simplicial tree on the vertex set $[n].$ Suppose that $I(\Delta)$ has a linear resolution, equivalently \cite[Theorem 3.17]{Zheng} $\Delta$ satisfy the intersection property. Then, there exists an ordering on the facets of $\Delta$, say $F_1,\ldots,F_r$, such that
\begin{itemize}
    \item $F_1,\ldots,F_r$ is a good leaf ordering on the facets of $\Delta;$
    \item $\dist_{\Delta}(F_i,F_{i+1}) =1$ for all $1 \le i \le r-1.$
\end{itemize}
\end{lemma}
\begin{proof}
 The case when the number of facets of $\Delta$ is one trivially holds. Assume that the number of facets of $\Delta$ is more than one. It follows from  \cite[Corollary 3.4]{HHTZ} that a good leaf exists, say $F_r$, in $\Delta$, where $r$ is the number of facets of $\Delta$. Set $\Delta_1=\Delta\setminus \{F_r\}.$ Since $I(\Delta)$ has a linear resolution, by \cite[Proposition 3.9]{Zheng}, $\Delta$ is pure and connected in codimension one.  Therefore, by \cite[Corollary 1.15]{Zheng}, $\Delta_1$ is a pure simplicial tree connected in codimension one, and consequently, by \cite[Lemma 3.11]{Zheng},  $I(\Delta_1)$ has a linear resolution. We claim that there exists a facet $F_{r-1}$ of $\Delta_1$ such that 
\begin{itemize}
    \item $F_{r-1}$ has a free vertex in $\Delta_1$;
    \item $\dist_{\Delta}(F_{r-1},F_{r}) =1$.
\end{itemize} 
Let $H_1, \ldots, H_k$ be the facets of $\Delta$ such that $\dist_{\Delta}(F_r, H_i) =1$ for all $1 \le i \le k.$ If possible, we assume that the claim does not hold, i.e., for each $1 \le i \le k $, $H_i$ has no free vertex in $\Delta_1.$ Since $\dist_{\Delta}(F_r,H_i) =1$, $|F_r \cap H_i|=|F_r|-1$ for each $1 \le i \le k.$ Consequently, $F_r \cap H_i=F_r\cap H_j=H_i\cap H_j$ for each $1 \le i<j \le k$ as $F_r$ is a good leaf of $\Delta.$ Now, for each $1 \le i \le k$, there exists a vertex $x_i \in H_i \setminus F_r$ such that $H_i=\{x_i\} \cap (H_i \cap F_r).$ It is easy to note that $x_1,\ldots, x_k$ are distinct. Since $H_i$ has not a free vertex in $\Delta_1$, there exists a facet $M_i$ of $\Delta_1$ such that $x_i \in H_i \cap M_i.$ With the help of \cite[Corollary 1.13]{Zheng} we can even assume that $|M_i \cap H_i|=|H_i|-1$. By construction, $F_r,H_i,M_i$ is a proper chain in $\Delta. $ Since $F_r$ is a good leaf in $\Delta$, $M_i \cap F_r \subseteq H_i \cap F_r$. Observe that $|M_i \cap F_r| <|F_r|-1$ as if $|M_i \cap F_r| =|F_r|-1$, then  $M_i \cap F_r = H_i \cap F_r$ which implies that $M_i=H_i$. Consequently, $F_r,H_i,M_i$ is a proper irredundant chain in $\Delta$ and $|M_i\cap F_r|=|F_r|-2$ which implies that  $M_i\cap F_r =M_j \cap F_r$ for each $1 \le i,j \le k$. Also, there exists $z_i \in (F_r \cap H_i)$ such that $z_i \not\in M_i$. Now, we have the following cases:
\begin{enumerate}
    \item[Case 1.] Suppose that $k\ge 2.$  Then, $M_1,H_1, H_2,M_2$ is a proper chain in $\Delta_1. $ Suppose that $|M_1 \cap H_2|\ge |H_2|-1$. As $x_1 \in M_1 \setminus H_2$,  $M_1=(M_1\cap H_2) \cup \{x_1\}$. This will  impose that $x_2 \in M_1$, which further implies that the collection of facets $\{M_1, H_1, H_2\}$ of $\Delta_1$ has not a leaf. We have a contradiction. Therefore, $x_2 \not\in M_1$ and $|M_1 \cap H_2|<|H_2|-1.$ Similarly, $x_1 \not\in M_2$ and $|H_1 \cap M_2|<|M_2|-1.$ 

  Observe that $z_1,z_2,x_1,x_2 \not\in M_1 \cap M_2$ and $z_1,z_2 \not\in \{x_1,x_2\}$. Consequently, $|M_1 \cap M_2| <|M_2|-1.$ This conclude that  $M_1,H_1, H_2,M_2$ is a proper irredundant chain in $\Delta_1.$ Therefore, $\dist_{\Delta_1}(M_1,M_2)=3.$ Since $I(\Delta_1)$ has a linear resolution, $\Delta_1$ satisfies intersection property which implies that  $|M_1 \cap M_2|= |M_1|-3.$ Since $M_1 \cap F_r =M_2 \cap F_r \subseteq M_1 \cap M_2$, we get that $|M_1 \cap M_2| \ge |M_2|-2$ which is a contradiction. 
  \item[Case 2.] Suppose that $k=1.$ Since $z_1 \in H_1\setminus M_1$ and $H_1$ has no free vertex in $\Delta_1$, there exists a facet of $N$ of $\Delta_1$ such that $z_1\in H_1 \cap N.$ Without loss of generality, by \cite[Corollary 1.13]{Zheng} we can assume that $|H_1\cap N|=|H_1|-1.$ In $\Delta$, we have $N\cap F_r \subseteq H_1\cap F_r$, as $F_r$ is a good leaf of $\Delta.$ Now, $F_r,H_1,N$ is a proper irredundant chain in $\Delta.$ Therefore, $\dist_{\Delta}(F_r,N)=2.$ Observe that $F_r\cap N \not\subset F_r\cap M_1$ as $z_1\in (F_r\cap N)\setminus M_1.$ Therefore, $F_r\cap M_1 \subseteq F_r\cap N$ which implies that $(F_r\cap M_1)\sqcup \{z_1\} \subseteq F_r\cap N$. Thus, $|F_r\cap N|= |F_r|-1$. Since $I(\Delta)$ has linear resolution, $\dist_{\Delta}(F_r,N)=1$ which is contradiction. 
\end{enumerate} Thus, in both scenarios, we have a contradiction. Thus, for some  $1 \le i \le k$, $H_i$ has a free vertex in $\Delta_1$. Assume that $H_1$ has a free vertex in $\Delta_1$ without loss of generality. Next, we claim that $H_1$ is a good leaf of $\Delta_1$. Let $w \in H_1$ be a free vertex of $H_1$ in $\Delta_1$. Now, we have the following two cases:
\begin{enumerate}
\item[Case a.] Suppose that $w \not\in F_r.$ Then, $w=x_1.$ For any facet, $P$ of $\Delta_1$ such that $P \ne H_1$, $P\cap H_1=P \cap (\{x_1\} \sqcup (H_1\cap F_r))=(P\cap \{x_1\}) \sqcup (P\cap  H_1 \cap F_r) =P\cap H_1 \cap F_r.$ This implies that $P\cap H_1=P \cap F_r$.  Let $A, B$ be any two facets of $\Delta_1$ such that $A,B \not\in \{H_1\}.$ In $\Delta$, either $A\cap F_r \subseteq B\cap F_r$ or $B\cap F_r \subseteq A\cap F_r$ as $F_r$ is a good leaf of $\Delta$ which implies that either $A\cap H_1 \subseteq B\cap H_1$ or $B\cap H_1 \subseteq A\cap H_1$. Thus, $H_1$ is a good leaf of $\Delta_1.$
\item[Case b.] Suppose that $w \in F_r.$ If possible, assume that $H_1$ is not a leaf of $\Delta_1.$ Then, there exist two facets $A, B$ of $\Delta_1$ other than $H_1$ such that $A\cap H_1 \not\subset B\cap H_1$ and $B \cap H_1 \not\subset A \cap H_1.$ Since $F_r$ is a leaf in $\Delta,$ assume, without loss of generality that $ A\cap F_r \subseteq B\cap F_r.$ Observe that $x_1 \in A $ and $x_1\not\in B$ as if $x_1 \notin A$ or $x_1 \in B$, then $A\cap H_1\subseteq B\cap H_1$. Therefore, $B\cap H_1=B\cap F_r \neq \emptyset$ and $|B\cap H_1|\le |H_1|-2$. Let $|B\cap F_r|=|B\cap H_1|=|F_r|-j$ for some $j\ge 2.$ Since $\Delta_1$ satisfies intersection property, there exists a unique proper irredundant chain of length $j$, say $H_1=B_0,B_1,\ldots, B_j=B$. Next, note that $F_r,H_1=B_0,B_1,\ldots,B_j=B$ is a proper irredundant chain in $\Delta $ between $F_r$ and $B.$ As $\Delta $ satisfies intersection property $j+1=\dist_{\Delta}(F_r,B)=|F_r|-|B\cap F_r|=j$ which is a contradiction. Therefore, $H_1$ is a good leaf of $\Delta_1$. 
\end{enumerate}
Take $F_{r-1}=H_1$ is a good leaf of $\Delta_1=\Delta\setminus \{F_r\}.$ Now, using the claim successively, we get a desired ordering of facets of $\Delta.$ 
\end{proof}
\begin{lemma}\label{lemo}
Let $\Delta$ be a simplicial tree on the vertex set $[n]$ such that $I(\Delta)$ has a linear resolution, equivalently \cite[Theorem 3.17]{Zheng} $\Delta$ satisfy the intersection property. Then, the ordering given in Lemma \ref{lemmag} also satisfies the following:
\begin{enumerate}[\rm(a)]
    \item if there exists $x \in F_j \setminus F_i$ for some $j<i$, then $x \notin F_k$ for all $k \geq i$;
    \item for any $j<i$, there exists $k \in [j,i-1]$ such that $|F_k \cap F_i|=|F_i|-1$ and $F_j \cap F_k \not\subset F_i$.
\end{enumerate}
\end{lemma}
\begin{proof}
(a)   We use  induction on $k \ge i$. The base case $k=i$ is trivial. Assume that
$k > i$ and $x \notin F_{k-1}$. Suppose, on the contrary, that $x \in F_k$. Then $x \in F_j \cap  F_k$. Since $\dist_{\Delta}(F_{k-1},F_k) =1$, $|F_{k-1} \cap F_k|=|F_k|-1$. ALso, as $F_k$ is a good leaf of $\langle F_1,\ldots,F_{k-1}\rangle$, therefore $F_j \cap  F_k \subseteq F_{k-1} \cap F_k$. Therefore, $x \in F_{k-1}$, a contradiction. 
\par (b) For $j<i$, it is easy to note that $F_j,F_{j+1},\ldots,F_i$ is a proper chain between $F_j$ and $F_i$. Now, after removing suitable facets from this proper chain, we obtain an irredundant proper chain,  $F_j,F_{j_1},\ldots,F_{j_{l-1}},F_{j_l}=F_i$. If $\dist_{\Delta}(F_j,F_i)=1$, then take $k=j.$ Or else we have $
    \dist_{\Delta}(F_j,F_{j_{l-1}}) = \dist_{\Delta}(F_j,F_i)-1>0.$
Consequently, since $\Delta$ satisfy  intersection property, $|F_j \cap F_{j_{l-1}}|=|F_j \cap F_i|+1.$ 
Suppose now that $F_j \cap F_{j_{l-1}} \subseteq F_i$. Then, $F_j \cap F_{j_{l-1}} \subseteq F_i \cap F_j$ which is a contradiction to the fact that $|F_j \cap F_{j_{l-1}}|=|F_j \cap F_i|+1.$ Hence,  $F_j \cap F_{j_{l-1}} \not\subset F_i$.
\end{proof}

\begin{lemma}\label{main-lemma}
Let $\Delta$ be a simplicial tree on the vertex set $[n]$ such that $I(\Delta)$ has a linear resolution, equivalently \cite[Theorem 3.17]{Zheng} $\Delta$ satisfy the intersection property. Then, $I(\Delta)^s$ has linear quotients for all $s\ge 1.$
\end{lemma}
\begin{proof}
By Lemma \ref{lemmag}, there exists an ordering of facets of $\Delta$, say $F_1,\ldots,F_r$, such that
\begin{itemize}
    \item $F_1,\ldots,F_r$ is a good leaf ordering on the facets of $\Delta;$
    \item $\dist_{\Delta}(F_i,F_{i+1}) =1$ for all $1 \le i \le r-1.$
\end{itemize} We fix that $m_1>\cdots >m_r$, where $m_i=\prod\limits_{j\in F_i}x_j$ for $1 \leq i \leq r$. One could easily verify that for any $s \geq 1$, any minimal monomial generator $M$ of $I(\Delta)^s$ has the unique expression $M=m_1^{a_1}m_2^{a_2}\cdots  m_r^{a_r}$ with $\sum\limits_{i=1}^ra_i=s$.  Now, for any two minimal monomial generators $M=m_1^{b_1}m_2^{b_2}\cdots m_r^{b_r}, N=m_1^{a_1}m_2^{a_2}\cdots m_r^{a_r}$ of $I(\Delta)^s$,
 we say $M > N$ if and only if $(b_1,\ldots,b_r) >_{lex} (a_1,\ldots,a_r)$. In this way, we get a total order among the minimal monomial generators of $I(\Delta)^s.$
 We claim that with respect to this total order among the minimal monomial generators of $I(\Delta)^s$ has linear quotients. By \cite[Lemma 8.2.3]{Herzog'sBook}, it is sufficient to prove that for any  two minimal monomial generators $M>N$ of $I(\Delta)^s$, there exists a minimal monomial generator $P>N$ such that $(P):N$ is generated by a variable and $(M):N \subseteq (P):N.$

Let $M=m_1^{b_1}\cdots m_r^{b_r}$ and $N=m_1^{a_1}\cdots m_r^{a_r}$ be two minimal monomial generators of $I(\Delta)^s$ such that $M>N.$ Therefore,  $a_t,b_t \ge 0$ for all $1 \le t \le r$ and $\sum\limits_{t=1}^ra_t=\sum\limits_{t=1}^rb_t=s$. We set 
$p =\min\{ t \in [r] \;: b_t-a_t>0  \}$ and
$q =\min\{t \in [r]: b_t-a_t<0\}.
$ Since $(b_1,\ldots,b_r) >_{lex} (a_1,\ldots,a_r)$, we get $p<q.$ Observe that 
$ b_p > a_p$, $b_i \geq a_i$ for all $p <i<q,$ and $b_q<a_q.$

 By Lemma \ref{lemo} (b), there exists $k \in [p,q)$ such that $F_p \cap F_k \not\subset F_q$ and $|F_k \cap F_q|=|F_q|-1.$ Let $x \in (F_p \cap F_k) \setminus F_q.$ We define 
 $c=(c_1,\ldots,c_r)$ as follows:
$$c_t=
\begin{cases}
 a_t & ~\text{if}~ t \notin \{q,k\} \\
 a_{t}+1  & ~\text{if} ~ t=k \\ 
 a_{t}-1  & ~\text{if}  ~t=q.
 \end{cases}
$$
Clearly, $(c_1,\ldots,c_r)>_{lex} (a_1,\ldots,a_r),$ and $\sum\limits_{t=1}^rc_t=s.$ Take $P=m_1^{c_1}\cdots m_r^{c_r}$. It is easy to verify that $(P):N=(m_k):m_q=(x).$  Consider,  \begin{eqnarray*}
(M):N &=& (m_p^{b_p}\cdots m_r^{b_r}):m_p^{a_p} \cdots m_r^{a_r} \\
&=& m_p^{b_p-a_p}\cdots m_{q-1}^{b_{q-1}-a_{q-1}} \left(\frac{m_q^{b_q}\cdots m_r^{b_r}}{\gcd(m_q^{b_q}\cdots m_r^{b_r}, m_q^{a_q}\cdots m_r^{a_r})}  \right). 
\end{eqnarray*}
Since $x \in (F_p \cap F_k) \setminus F_q$ and $b_p>a_p$, using Lemma \ref{lemo} (a) we obtain $(M):N \subseteq (x)=(P):N.$ Hence, the assertion follows.
\end{proof}

We conclude this section by characterizing the linearity of the minimal free resolution of facet ideals of simplicial trees.

\begin{theorem}\label{linear}
Let $\Delta$ be a simplicial tree on the vertex set $[n].$ Then, the followings are equivalent:
\begin{enumerate}
\item $\Delta$ satisfies the  intersection property.
    \item $I(\Delta)$ has a linear resolution.
    \item $I(\Delta)^s$ has linear quotients for all $s \ge 1$ and $\Delta$ is pure.
    \item $I(\Delta)^s$ has linear quotients for some $s$ and $\Delta $ is pure.
    \item $I(\Delta)^s$ has  a linear resolution for some $s.$
    \item $I(\Delta)^s$ has linear first syzygies for some $s.$
\end{enumerate}
\end{theorem}
\begin{proof}
It follows from \cite[Proposition 3.17] {Zheng} that $(1)$ and $(2)$ are equivalent.  The implication $(2) \implies (3)$ follows from Lemma \ref{main-lemma} and the immplications $(3) \implies (4),$  and $(5) \implies (6)$ are immediate. Note that if $\Delta$ is pure, then for all $s \ge 1$, $I(\Delta)^s$ is generated in the same degree. Consequently, by \cite[Proposition 8.2.1]{Herzog'sBook}, the implication $(4) \implies (5)$ follows. Next, we prove $(6) \implies (2)$. Since $I(\Delta)^s$ has linear first syzygies, $I(\Delta)^s$ is generated in the same degree. Therefore, $\Delta$ is pure.  Suppose that $I(\Delta)$ has no linear resolution. By \cite[Proposition 3.9 (i)]{Zheng}, $I(\Delta)$ can not have linear quotients with respect to any order on minimal monomial generators. Let $F_1, \ldots, F_r$ be a good leaf order on the facets of $\Delta.$ Then, for some $k \le r-1$, $(m_1,\ldots, m_k):m_{k+1}$ is not generated by monomials of degree one. Since $F_{k+1}$ is a good leaf of the subcomplex with facets $F_1,\ldots, F_{k+1},$ $\left\{F_i \cap F_{k+1} ~ \mid ~ 1 \le i \le k  \right\}$ is a total order set, i.e., there exist distinct $i_1, \ldots, i_k \in \{ 1, \ldots, k\}$ such that $F_{i_1} \cap F_{k+1} \subseteq \cdots \subseteq F_{i_k} \cap F_{k+1}$. Let $t \in \{ 1, \ldots, k\}$ be the largest such that 
\begin{itemize}
    \item $(m_{i_t}) : m_{k+1}$ is not generated by a monomial of degree one;
    \item $(m_{i_j} ~\mid ~ j>t): m_{k+1}$ is generated by monomials of degree one;
    \item $(m_{i_t}) : m_{k+1} \not\subset (m_{i_j}) : m_{k+1} $ for each $j>t.$ 
\end{itemize}
 We claim that $F_{i_j} \not\subseteq F_{i_t} \cup F_{k+1}$ for each $j \in \{1, \ldots, k\} \setminus \{t\}.$ For each $j>t,$ there exists a variable $z_j \in F_{i_j} \setminus F_{k+1}$ such that $z_j \notin F_{i_t}$ as $(m_{i_t}) : m_{k+1} \not\subset (m_{i_j}) : m_{k+1} $. Consequently, $F_{i_j} \not\subseteq F_{i_t} \cup F_{k+1}$ for each $j>t.$ Suppose that for some $j<t$, $F_{i_j} \subseteq F_{i_t} \cup F_{k+1}.$
Since $F_{i_j} \cap F_{k+1} \subseteq F_{i_t} \cap F_{k+1} \subseteq F_{i_t}$, we get that $F_{i_j} \subseteq F_{i_t}$ which is not possible as $\Delta$  is pure and  $F_{i_j}, F_{i_t}$ are distinct facets. Thus, the claim follows. Let $\Delta'$ be the induced subcomplex of $\Delta$ on the vertex set $F_{i_t} \cup F_{k+1}.$  It follows from the above claim that $I(\Delta')=(m_{i_t}, m_{k+1}).$ Set $u=\prod\limits_{l \in F_{i_t}\cap F_{k+1}} x_l$, $v=\prod\limits_{l \in F_{i_t}\setminus F_{k+1}} x_l$ and $w=\prod\limits_{l \in F_{k+1}\setminus F_{i_t}} x_l$.  Observe that $v, w$ are monomials of degree at least two, as $(m_{i_t}) : m_{k+1}$ is not generated by a monomial of degree one. Since $v$ and $w$ are monomials in a disjoint set of variables, it follows from \cite[Theorem 2.1]{GVa05} that  $\beta_{1,j}^R\left((v,w)^s\right) \neq 0$, where $j=\deg\left((vw)^s\right).$  Consequently, $\beta_{1,j}^R\left(I(\Delta')^s\right)=\beta_{1,j}^R\left((uv,uw)^s\right) \neq 0$, where $j=\deg\left((uvw)^s\right).$ Since $\Delta'$ is an induced subcomplex of $\Delta$, it follows from \cite[Theorem 3.6]{BCDMS} that $\beta_{1,j}^R\left(I(\Delta)^s\right) \neq 0$ with $j=\deg\left((uvw)^s\right).$ This is a contradiction to the fact that $I(\Delta)^s$ has linear first syzygies for some $s.$ Thus, $(6) \implies (2)$ follows. Hence, the assertion follows.  
\end{proof}

\section{Regularity of $t$-path ideals of rooted trees }\label{section4}
In this section, we give an explicit formula for the regularity of $t$-path ideals of perfect rooted trees with $\left\lceil{\frac{\hgt(\Gamma)+1}{2}}\right\rceil \leq t \leq \hgt(\Gamma)+1$. We first recall the notion of rooted trees and $t$-path ideals associated with them.

A \textit{rooted tree} $(\Gamma,x_0)$ is a tree with a distinguished vertex $x_0$, called the \textit{root} of $\Gamma$. In this article, whenever we mention a rooted tree, we assume it is a directed tree in which edges are implicitly directed away from the root. A \textit{directed path of length $(t-1)$} is a sequence of distinct vertices $x_{i_1} ,\ldots, x_{i_t},$ such that $(x_{i_j}
, x_{i_{j+1}})$
is a directed edge from $x_{i_j}$ to $x_{i_{j+1}}$ for all $j = 1,\ldots,t-1$. The \textit{$t$-path ideal of $\Gamma,$} denoted $I_t(\Gamma),$ is the ideal
$$(x_{i_1} \cdots x_{i_t} : ~x_{i_1} ,\ldots, x_{i_t}~  \text{is~ a~ directed~ path~
on }~ t~ \text{ vertices in}~ \Gamma).$$

\begin{definition}
Let $(\Gamma,x_0)$ be a rooted tree on the vertex set $\{ x_0,..., x_n\}$. 
\begin{enumerate}[a)]
\item For vertices $x,y$, the distance between $x$ and
$y$, denoted by $d_{\Gamma}(x,y),$ is the length of the unique directed path from $x$ to $y$ in $\Gamma$. If there is no directed path from  $x$ to $y$, we set $d_{\Gamma}(x,y)= \infty.$
\item The outdegree of a vertex $x$ in $\Gamma$, denoted $\deg_{\Gamma}^{+}(x),$ is the number of edges directed away from $x$.
\item If  $\deg_{\Gamma}^{+}(x)=0$, then $x$ is called a leaf of $\Gamma$. 
\item The parent of a vertex $x$ in $\Gamma$ is the vertex that is immediate before  $x$ on the unique directed path from $x_0$ to $x$.
\item A vertex $u$ is called a  descendant of $x$ if there is a directed path from $x$ to $u$ of length at least one. 
\item The level of a vertex $x,$ denoted $\ell_{\Gamma}(x),$ is $d_{\Gamma}(x_0,x)$. The height of $\Gamma$, denoted by $\hgt(\Gamma)$, is defined to be $\max_x\ell_{\Gamma}(x).$
\item If  $\deg_{\Gamma}^{+}(x)=k$ for all vertices $x$ with $\ell_{\Gamma}(x) \leq \hgt(\Gamma)-1$, then $(\Gamma,x_0)$ is called a $k$-nary rooted tree.
\item A rooted forest is a disjoint union of rooted trees. Further, the $t$-path ideal of a rooted forest is the sum of the $t$-path ideals of rooted trees.
\item  The level of $x$ in a rooted forest $T$ is defined to be the level of $x$ inside the rooted tree containing $x$. The height of $T$ is the largest height among all the rooted trees of $T$.
\item An induced subtree of $\Gamma$ is a directed tree which is an induced subgraph of of $\Gamma$.
\item  For a vertex $z$ of $\Gamma$,  the induced rooted  subtree of $\Gamma$ with root $z$ is a rooted tree on the vertex set $\{z\} \cup\{x : x$ is a descendant of $z\}.$
\item For any induced subtree $\Gamma'$ of $\Gamma$, by $\Gamma \setminus \Gamma'$  we denote the induced subforest of $\Gamma$ obtained by removing the vertices of $\Gamma'$ and the edges incident to these vertices.
\item A rooted tree $(\Gamma,x_0)$ is called a {\it perfect rooted tree} if $\ell_{\Gamma}(x) = \hgt(\Gamma)$ for all leaf $x$ in $\Gamma.$ 
\end{enumerate}


\end{definition}

Authors in \cite{RT} gave an explicit formula for the regularity of $I_{\hgt(\Gamma)+1}(\Gamma)$. In the following theorem, we obtain a different formula for the regularity of $I_{\hgt(\Gamma)+1}(\Gamma)$ in terms of the outdegree of vertices which will be useful in proving various other important results in this paper.
\begin{theorem}\label{th1}
Let $(\Gamma,x_0)$ be a perfect rooted tree with $\hgt{(\Gamma)}\geq 1,$ and  $t=\hgt({\Gamma})+1$. Then,
$$\displaystyle\reg\left(\frac{R}{I_t(\Gamma)}\right)= 1 + \sum_{\ell_{\Gamma}(x) = 0}^{\hgt({\Gamma})-2}\deg_{\Gamma}^+(x).$$
\end{theorem}
\begin{proof}
 We use induction on $\hgt(\Gamma)$. Suppose that $\hgt(\Gamma)=1$. Then, $$I_t(\Gamma) =(x_0y\; :\;  \ell_{\Gamma}(y)=1)=x_0(y \; : \;  \ell_{\Gamma}(y)=1).$$ Thus, $\displaystyle\reg\left(\frac{R}{I_t(\Gamma)}\right)= 1+ \reg\left( \frac{R}{(y \; : \; \ell_{\Gamma}(y)=1)}\right)= 1$, and hence, the result is true for  $\hgt(\Gamma)=1$. Assume now that  $\hgt(\Gamma)>1$. Consider the following short exact sequence:
$$ \displaystyle 0 \rightarrow \frac{R}{I_t(\Gamma):x_0} (-1)\xrightarrow{\cdot x_0} \frac{R}{I_t(\Gamma)} \rightarrow \frac{R}{I_t(\Gamma)+(x_0)}  \rightarrow 0.$$ Since $t=\hgt(\Gamma)+1,$ $x_0 $ divides each monomial of $I_t(\Gamma)$, and thus, $I_t(\Gamma)+(x_0)=(x_0).$ Consequently, we have $\displaystyle\reg\left(\frac{R}{I_t(\Gamma)+(x_0)}\right)=0<t-1 \leq \reg\left(\frac{R}{I_t(\Gamma)}\right).$ Therefore, by \cite[Lemma 2.10]{HHJ13}, $\displaystyle \reg\left(\frac{R}{I_t(\Gamma)}\right)=\reg\left(\frac{R}{I_t(\Gamma):x_0}(-1)\right)=1+\reg\left(\frac{R}{I_t(\Gamma):x_0}\right).$  Observe that $I_t(\Gamma):x_0=I_{t-1}(\Gamma \setminus \{x_0\})$. Let  $x_{i_1},\ldots,x_{i_k}$ be descents of $x_{0}$ in $\Gamma$. Then $\Gamma \setminus \{x_0\}$ is the disjoint union of perfect rooted trees, say, $(\Gamma_1,x_{i_1}),\ldots, (\Gamma_k,x_{i_k})$ each of height $\hgt(\Gamma)-1.$ Thus, $\displaystyle \reg\left(\frac{R}{I_t(\Gamma)}\right)=1+\reg\left(\frac{R}{I_t(\Gamma):x_0}\right)=1+\sum\limits_{j=1}^k\reg\left(\frac{R}{I_{t-1}(\Gamma_j)}\right).$ Since for each $j$, $(\Gamma_j,x_{i_j})$ is a perfect rooted tree with $t-1=\hgt(\Gamma)=\hgt(\Gamma_j)+1$, by induction, $$\reg\left(\frac{R}{I_{t-1}(\Gamma_j)}\right) = 1+ \sum_{\ell_{\Gamma_j}(x)=0}^{\hgt(\Gamma_j)-2} \deg_{\Gamma_j}^+(x).$$  Therefore,
\begin{eqnarray*}
\reg\left(\frac{R}{I_t(\Gamma)}\right)&=& 1+\sum\limits_{j=1}^k\reg\left(\frac{R}{I_{t-1}(\Gamma_j)}\right)= 1+\sum_{j=1}^k \left( 1+ \sum\limits_{\ell_{\Gamma_j}(x)=0}^{\hgt(\Gamma_j)-2} \deg_{\Gamma_j}^{+}(x) \right)\\&=& 1+k+\sum\limits_{j=1}^k\sum\limits_{\ell_{\Gamma_j}(x) = 0}^{\hgt({\Gamma_j})-2}\deg_{\Gamma_j}^+(x). 
\end{eqnarray*} Note that for $x \in V(\Gamma \setminus \{x_0\})$ and for each $j$, $\ell_{\Gamma_j}(x) =\ell_{\Gamma}(x)-1$ and $\hgt(\Gamma_j) =\hgt(\Gamma)-1.$ Thus, 
\begin{eqnarray*}
\reg\left(\frac{R}{I_t(\Gamma)}\right)=  1+k+\sum\limits_{\ell_{\Gamma}(x)=1}^{\hgt(\Gamma)-2} \deg_{\Gamma}^{+}(x) = 1+\sum\limits_{\ell_{\Gamma}(x)=0}^{\hgt(\Gamma)-2} \deg_{\Gamma}^{+}(x) 
\end{eqnarray*} as $\deg_{\Gamma}^+(x_0)=k.$ Hence, the assertion follows.
\end{proof}

 \begin{theorem} \label{thm2}
Let $(\Gamma,x_0)$ be a perfect rooted tree with $\hgt(\Gamma) \geq 1$,  and $t=\hgt(\Gamma).$ Then $$\displaystyle\reg\left(\frac{R}{I_t(\Gamma)}\right)=   \sum_{\ell_{\Gamma}(x) = 0}^{\hgt({\Gamma})-2}\deg_{\Gamma}^+(x).$$
 \end{theorem}
 \begin{proof}  Consider the following short exact sequence:
$$ \displaystyle 0 \rightarrow \frac{R}{I_t(\Gamma):x_0} (-1)\xrightarrow{\cdot x_0} \frac{R}{I_t(\Gamma)} \rightarrow \frac{R}{I_t(\Gamma)+(x_0)} \rightarrow 0.$$ Note that $I_t(\Gamma)+(x_0)=I_t(\Gamma \setminus \{ x_0\})+(x_0).$ Let $x_{i_1},\ldots,x_{i_k}$ are descents of $x_{0}$ in $\Gamma.$ Then $\Gamma \setminus \{x_0\}$ is the disjoint union of perfect rooted trees, say, $(\Gamma_1,x_{i_1}),\ldots, (\Gamma_k,x_{i_k})$ each of height $\hgt(\Gamma)-1.$ Since for each $j$, $(\Gamma_j,x_{i_j})$ is a perfect rooted tree and $t =\hgt(\Gamma)=\hgt(\Gamma_j)+1$, by Theorem \ref{th1},  $\reg\left(\frac{R}{I_t(\Gamma_j)}\right)=  1+\sum\limits_{\ell_{\Gamma_j}(x) = 0}^{\hgt({\Gamma_j})-2}\deg_{\Gamma_j}^+(x)$. Thus,
\vspace{-0.2 cm}
\begin{eqnarray*}
\reg\left(  \frac{R}{I_t(\Gamma)+(x_0)}  \right)
&=&  \sum\limits_{j=1}^k\reg\left(\frac{R}{I_{t}(\Gamma_j)}\right) = \sum\limits_{j=1}^{k}\left( 1+\sum\limits_{\ell_{\Gamma_j}(x) = 0}^{\hgt({\Gamma_j})-2}\deg_{\Gamma_j}^+(x) \right)\\
&=& k+\sum\limits_{j=1}^k\sum\limits_{\ell_{\Gamma_j}(x) = 0}^{\hgt({\Gamma_j})-2}\deg_{\Gamma_j}^+(x) = \sum_{\ell_{\Gamma}(x) = 0}^{\hgt({\Gamma})-2}\deg_{\Gamma}^+(x).
\end{eqnarray*}
Next, we claim that  $I_t(\Gamma):x_0=I_{t-1}(\Gamma \setminus \{x_0,z:\ell_{\Gamma}(z)=\hgt(\Gamma)\}).$ Note that $$I_t(\Gamma)= I_t(\Gamma \setminus \{ z \; : \; \ell_{\Gamma}(z) =\hgt(\Gamma)\}) + I_t(\Gamma \setminus \{x_0\}).$$ Therefore, \begin{align*}
    I_t(\Gamma):x_0 &=I_t(\Gamma \setminus \{ z \; : \; \ell_{\Gamma}(z) =\hgt(\Gamma)\}):x_0 + I_t(\Gamma \setminus \{x_0\}) :x_0\\& = I_{t-1}(\Gamma \setminus \{x_0, \;  z \; : \; \ell_{\Gamma}(z) =\hgt(\Gamma)\}) + I_t(\Gamma \setminus \{x_0\}) \\&= I_{t-1}(\Gamma \setminus \{ x_0, \; z \; : \; \ell_{\Gamma}(z) =\hgt(\Gamma)\}), 
\end{align*} where the last equality follows from the fact that any path of length $t$ in $\Gamma \setminus \{x_0\}$ contains a path of length $t-1$ in $\Gamma \setminus \{ x_0, \; z \; : \; \ell_{\Gamma}(z) =\hgt(\Gamma)\}$.  For  $ 1 \leq j \leq k$, set $\Gamma'_j=\Gamma_j \setminus \{z:\ell_{\Gamma_j}(z)=\hgt(\Gamma_j)\}.$ Then,  $\Gamma \setminus \{ x_0, \; z \; : \; \ell_{\Gamma}(z) =\hgt(\Gamma)\}$ is the disjoint union of perfect rooted trees $(\Gamma'_1,x_{i_1}),\ldots, (\Gamma'_k,x_{i_k})$ and $\hgt(\Gamma_j')=\hgt(\Gamma_j)-1=\hgt(\Gamma)-2$ for $1 \leq j \leq k$. 
Therefore, by Theorem \ref{th1}, 
\begin{eqnarray*}
\reg\left( \frac{R}{I_t(\Gamma):(x_0)}  \right)
&=& \sum\limits_{j=1}^k\reg\left(\frac{R}{I_{t-1}(\Gamma'_j)}\right) = \sum\limits_{j=1}^{k}\left( 1+\sum\limits_{\ell_{\Gamma_j'}(x) = 0}^{\hgt({\Gamma_j'})-2}\deg_{\Gamma_j'}^+(x) \right)\\
&=& k+\sum\limits_{j=1}^k\sum\limits_{\ell_{\Gamma_j'}(x) = 0}^{\hgt({\Gamma_j'})-2}\deg_{\Gamma_j'}^+(x) = \sum_{\ell_{\Gamma}(x) = 0}^{\hgt({\Gamma})-3}\deg_{\Gamma}^+(x)\\&<& \sum_{\ell_{\Gamma}(x) = 0}^{\hgt({\Gamma})-2}\deg_{\Gamma}^+(x)= \reg\left(\frac{R}{I_t(\Gamma) +(x_0)}\right).
\end{eqnarray*}
Thus, the desired result follows from \cite[Theorem 4.7(iii)]{GHJMNN19}.
 \end{proof}

As a consequence, we obtain the following corollary.
 \begin{corollary}
 Let $(\Gamma, x_0)$ be a $k$-nary rooted tree with $\hgt(\Gamma)\geq 1$.  
 \begin{enumerate}
     \item If $t=\hgt(\Gamma)+1,$ then $$\displaystyle\reg\left(\frac{R}{I_t(\Gamma)}\right)=\frac{k^{\hgt(\Gamma)}-1}{k-1}.$$
 \item If $t=\hgt(\Gamma),$ then $$\displaystyle\reg\left(\frac{R}{I_t(\Gamma)}\right)=\frac{k^{\hgt(\Gamma)}-k}{k-1}.$$
 \end{enumerate} \end{corollary}
  
In the following, we set the notations which we use throughout this article:
 \begin{notation}\label{nota}
 Let $(\Gamma,x_0)$ be a rooted tree with $\hgt(\Gamma) \geq t-1$, and  $z$ be a leaf such that  $\ell_{\Gamma}(z)=\hgt(\Gamma)$. Then, there exists a  unique path of length $t-1$ in $\Gamma$ that terminates at $z=x_t(z),$ say, $P(z):= x_1(z),\ldots,x_t(z)$. Let $x_{0}(z)$ be the parent of $x_{1}(z),$ if exists. 
For $j = 0,\ldots,t,$ let $\Gamma_j(z)$ be the induced  rooted subtree of $\Gamma$ rooted at $x_{j}(z),$ and let $\Delta^{\Gamma}_j(z) = \Gamma_j(z) \setminus (\Gamma_{j+1}(z) \cup  \{x_{j}(z)\}) $. Observe that if $\Gamma$ is perfect, then for each $ 0 \leq j \leq t-1$ $\Delta_{j}^{\Gamma}(z)$ is disjoint union of perfect rooted trees of height $t-j-1.$ Set
$ \Gamma(z) =
    \begin{cases}
      \Gamma \setminus \Gamma_0(z) & \text{if $x_0(z)$  exists}\\
     \Gamma \setminus \Gamma_1(z)  & \text{if $x_0(z)$  does not exist}
    \end{cases}.       
$ Note that in both cases $\Gamma(z)$ is either empty or a rooted tree. 
 \end{notation}

In the following, we collect a few results from \cite{BRT}, which will be useful in proving the main theorem of this section.
\begin{lemma}\label{lemma}
Let $(\Gamma,x_0)$ be a rooted tree with  $\hgt(\Gamma) \geq t-1$, and  $z$ be a leaf of $\Gamma$ such that  $\ell_{\Gamma}(z)=\hgt(\Gamma)$. 
Then, with Notation \ref{nota},\begin{enumerate}
\item \cite[Lemma 2.8]{BRT} \begin{align*}
    I_t(\Gamma \setminus \{z\}) : (x_1(z) \cdots x_t(z)) = I_t(\Gamma(z))+ (x_0(z))+\sum\limits_{j=0}^{t-1} I_{t-j}(\Delta_j^{\Gamma}(z)).
\end{align*}
    \item \cite[Corollary 3.3]{BRT} \begin{align*}
    \reg\left(\frac{R}{I_t(\Gamma)}  \right)=& \max \left\{\reg\left(\frac{R}{I_t(\Gamma \setminus \{z\})}\right),\reg\left(\frac{R}{I_t(\Gamma(z))}\right) + \sum\limits_{j=0}^{t-1}\reg\left(\frac{R}{I_{t-j}(\Delta^{\Gamma}_j(z) )}\right)+(t-1)\right\}. 
\end{align*}
\end{enumerate}
\end{lemma}

\begin{lemma}\label{lemm1}
 Let $(\Gamma,x_0)$ be a rooted tree with $\hgt(\Gamma) \geq 1,$ and let $\Gamma'$ be an induced rooted subforest of $\Gamma$. Then, $\reg\left(\frac{R}{I_t(\Gamma')}\right) \leq \reg\left(\frac{R}{I_t(\Gamma)}\right).$
\end{lemma}
\begin{proof}
The assertion follows from \cite[Lemma 2.5]{HaWood}.
\end{proof}

We now provide a procedure to study the regularity of $t$-path ideals of rooted trees. 
\begin{lemma} \label{lem2}
Let $(\Gamma, x_0)$ be a rooted tree with $\hgt(\Gamma) \geq t-1$. With Notation \ref{nota}, set $$\alpha(\Gamma) = \max_{ \ell_{\Gamma}(z)=\hgt(\Gamma)}\left\{\reg\left(\frac{R}{I_t(\Gamma(z))}\right)+ \sum_{j=0}^{t-1} \reg\left(\frac{R}{I_{t-j}(\Delta_j^{\Gamma}(z))}\right)+(t-1)\right\}.$$  Then,  \begin{align*}
    \reg\left(\frac{R}{I_t(\Gamma) } \right) \leq & \max \left\{\reg\left(\frac{R}{I_t(\Gamma')}\right), \alpha(\Gamma)
\right\},
\end{align*} where $\Gamma'=\Gamma\setminus \{ w \; : \; \ell_{\Gamma}(w)=\hgt(\Gamma)\}.$
\end{lemma}
\begin{proof}
For any rooted tree $(\Gamma, x_0)$, we set $n(\Gamma)=|\{ w \; : \; \ell_{\Gamma}(w) = \hgt(\Gamma)\}|.$ We use induction on $n(\Gamma)$. If $n(\Gamma)=1$,  then the result is true by Lemma \ref{lemma} and the fact that $\Gamma'=\Gamma \setminus \{z\}$ for the leaf $z$ with $\ell_{\Gamma}(z)=\hgt(\Gamma)$. Assume that $n(\Gamma)>1$ and the result is true for rooted trees having the number of leaves at the highest level strictly less than $n(\Gamma)$. Let $w$ be a leaf of $\Gamma$, which belongs to the highest level of $\Gamma$. Then, in the view of Lemma \ref{lemma}, it is enough to prove that \begin{align*}
 \reg\left(\frac{R}{I_t(\Gamma \setminus \{w\})}  \right) \leq  \max\left\{\reg\left(\frac{R}{I_t(\Gamma')}\right), \alpha(\Gamma)
\right\}  . 
\end{align*} Note that $n(\Gamma \setminus \{w\})< n(\Gamma)$ and $\hgt(\Gamma \setminus \{w\}) =\hgt(\Gamma).$ Therefore, by induction, we have \begin{align*}
    \reg\left(\frac{R}{I_t(\Gamma \setminus \{w\})}  \right) \leq & \max\left\{\reg\left(\frac{R}{I_t((\Gamma \setminus \{w\})\setminus \{z \; : \; \ell_{\Gamma \setminus \{w\}} (z)=\hgt(\Gamma \setminus \{w\}))}\right), \alpha(\Gamma \setminus \{w\})
\right\}\\ = & \max\left\{\reg\left(\frac{R}{I_t(\Gamma')}\right), \alpha(\Gamma \setminus \{w\})
\right\},
\end{align*} where $$\alpha(\Gamma \setminus \{w\}) = \max_{ \ell_{\Gamma \setminus \{w\}}(z)=\hgt(\Gamma \setminus \{w\})}\left\{\reg\left(\frac{R}{I_t((\Gamma \setminus \{w\})(z))}\right)+ \sum_{j=0}^{t-1} \reg\left(\frac{R}{I_{t-j}(\Delta_j^{\Gamma \setminus \{w\}}(z))}\right)+(t-1)\right\}.$$ Next, note that for each $z$ with $\ell_{\Gamma \setminus \{w\}}(z)=\hgt(\Gamma \setminus \{w\})$, $(\Gamma  \setminus \{w\})(z)$ and $\Delta^{\Gamma \setminus \{w\}}_j(z)$ are induced subtrees of $\Gamma (z)$ and $\Delta^{\Gamma }_j(z)$, respectively. Thus, by Lemma \ref{lemm1},
\begin{align*}
    \alpha(\Gamma \setminus \{w\}) &= \max_{ \ell_{\Gamma \setminus \{w\}}(z)=\hgt(\Gamma \setminus \{w\})}\left\{\reg\left(\frac{R}{I_t((\Gamma \setminus \{w\})(z))}\right)+ \sum_{j=0}^{t-1} \reg\left(\frac{R}{I_{t-j}(\Delta_j^{\Gamma \setminus \{w\}}(z))}\right)+(t-1)\right\}\\ & \leq \max_{ \ell_{\Gamma \setminus \{w\}}(z)=\hgt(\Gamma \setminus \{w\})}\left\{\reg\left(\frac{R}{I_t(\Gamma(z))}\right)+ \sum_{j=0}^{t-1} \reg\left(\frac{R}{I_{t-j}(\Delta_j^{\Gamma}(z))}\right)+(t-1)\right\}\\ & \leq \max_{ \ell_{\Gamma}(z)=\hgt(\Gamma)}\left\{\reg\left(\frac{R}{I_t(\Gamma(z))}\right)+ \sum_{j=0}^{t-1} \reg\left(\frac{R}{I_{t-j}(\Delta_j^{\Gamma}(z))}\right)+(t-1)\right\}\\ & =\alpha(\Gamma).
\end{align*}
Hence,  the desired result follows.
\end{proof}

\begin{remark}\label{clean}
    Let $(\Gamma,x_0)$ be a rooted tree and $I_t(\Gamma)$ be the path ideal of length $t-1$ with $t \geq 1$. Note that for a leaf $x$ at a level strictly less than $(t - 1),$ the generators of  $I_t(\Gamma \setminus \{x\})$ and $I_t(\Gamma)$ (in different polynomial rings) are same, and hence,  
$ \reg\left(\frac{R}{I_t(\Gamma \setminus \{x\})} \right)= \reg\left(\frac{R}{I_t(\Gamma) }\right)$. Therefore, we successively remove all leaves at levels strictly less than $(t-1)$. The
 rooted tree obtained after this process is called the {\it clean form} of $\Gamma,$ and it is  denoted by $C(\Gamma ).$ Also, note that $I_t(\Gamma)=I_t(C(\Gamma)).$
\end{remark}

\begin{remark}\label{clean1}
   Let $(\Gamma,x_0)$ be a perfect rooted tree of height $\hgt(\Gamma)\geq 2$ and let $t$ be a positive integer such that $\left\lceil{\frac{\hgt(\Gamma)+1}{2}}\right\rceil \leq t \leq \hgt(\Gamma)+1$. For a leaf $z$ in $\Gamma$, following Notation \ref{nota}, if   $\Gamma(z)$ is not a perfect rooted tree and $\hgt(\Gamma(z)) =\hgt(\Gamma)$, then we claim that $C(\Gamma(z))$ is a perfect rooted tree with $\hgt(C(\Gamma(z))) =\hgt(\Gamma)$.  Suppose that $x$ is a leaf in $\Gamma(z)$ such that $\ell_{\Gamma(z)}(x)< \hgt(\Gamma(z)).$ Then, following Notation \ref{nota}, $x$ is the parent of $x_0(z)$ which implies that $\ell_{\Gamma(z)}(x) = \hgt(\Gamma)-t-1$. Since $t \geq \left\lceil{\frac{\hgt(\Gamma)+1}{2}}\right\rceil >\frac{\hgt(\Gamma)}{2}$, we get $\ell_{\Gamma(z)}(x) = \hgt(\Gamma)-t-1<t-1$. Thus, by Remark \ref{clean}, $C(\Gamma(z))$ is a perfect rooted tree with $\hgt(C(\Gamma(z))) =\hgt(\Gamma)$.
\end{remark}
We now prove the main result of this section and its consequences.
\begin{theorem} \label{maint}
Let $(\Gamma,x_0)$ be a perfect rooted tree with $\hgt(\Gamma) \geq 1$ and let $t$ be a positive integer such that $\left\lceil{\frac{\hgt(\Gamma)+1}{2}}\right\rceil \leq t \leq \hgt(\Gamma) +1$.  Then $$ \reg\left(\frac{R}{I_t(\Gamma)}  \right) = \sum\limits_{\ell_{\Gamma}(x)=\hgt(\Gamma) -t}^{\hgt(\Gamma)-2}\deg^{+}_{\Gamma}(x).$$ 
\end{theorem}
\begin{proof}
First, we show that $ \reg\left(\frac{R}{I_t(\Gamma) } \right) \geq  \sum\limits_{\ell_{\Gamma}(x) = \hgt(\Gamma) -t}^{\hgt(\Gamma)-2}\deg^{+}_{\Gamma}(x).$ Let $U=\{w\in V(\Gamma) : \ell_{\Gamma}(w)=\hgt(\Gamma)-t\}$. For each $u \in U$, let $\Gamma_u$ denote the induced rooted subtree of $\Gamma$ rooted at $u$. Note that $\hgt(\Gamma_u)=t$ for each $u \in U.$ So, applying Theorem \ref{thm2}, we get $ \reg\left(\frac{R}{I_t(\Gamma_u)}  \right) = \sum\limits_{\ell_{\Gamma_u}(x)=0}^{\hgt(\Gamma_u)-2}\deg^{+}_{\Gamma_u}(x).$ Set 
$\Gamma_U=\bigsqcup\limits_{u \in U}\Gamma_u.$ Then, $\Gamma_{u} \cap \Gamma_{v}=\emptyset$ for distinct $u,v \in U$. Thus,
 \begin{eqnarray*}
\reg\left(\frac{R}{I_{t}(\Gamma_U)} \right)=\sum\limits_{u \in U} \reg\left(\frac{R}{I_{t}(\Gamma_u)} \right)
&=& \sum\limits_{u \in U}\sum\limits_{\ell_{\Gamma_u}(x)=0}^{\hgt(\Gamma_u)-2}\deg^{+}_{\Gamma_u}(x)= \sum\limits_{\ell_{\Gamma}(x)=\hgt(\Gamma)-t}^{\hgt(\Gamma)-2}\deg^{+}_{\Gamma}(x). 
\end{eqnarray*}  Since $\Gamma_U$ is an induced subforest of $\Gamma$, using  Lemma \ref{lemm1}, we get
$$\reg\left(\frac{R}{I_t(\Gamma)}\right) \geq  \sum\limits_{\ell_{\Gamma}(x)=\hgt(\Gamma) -t}^{\hgt(\Gamma)-2}\deg^{+}_{\Gamma}(x).$$  

Next we claim that $\reg\left(\frac{R}{I_t(\Gamma)}\right) \leq  \sum\limits_{\ell_{\Gamma}(x)=\hgt(\Gamma) -t}^{\hgt(\Gamma)-2}\deg^{+}_{\Gamma}(x).$  The proof proceeds by induction on $\hgt(\Gamma)+n(\Gamma),$ where $n(\Gamma)=|\{ w \; : \; \ell_{\Gamma}(w) = \hgt(\Gamma)\}|.$ If $\hgt(\Gamma)=1$, then $t=1$ or $t=2$. For $t=1$, the result is obvious. If $t=2$,  by convention, we have $\deg^{+}_{\Gamma}(x)=1,$ where  $\ell_{\Gamma}(x)=-1$, and thus, by Theorem \ref{th1}, $\reg\left(\frac{R}{I_t(\Gamma)}  \right) = \sum\limits_{\ell_{\Gamma}(x)=\hgt(\Gamma) -t}^{\hgt(\Gamma)-2}\deg^{+}_{\Gamma}(x).$ If $n(\Gamma)=1$, then $\Gamma$ is a directed path, and hence using \cite[Corollary 5.4]{BRT} we get $ \reg\left(\frac{R}{I_t(\Gamma)}  \right) =(t-1)\left\lceil \frac{n-t+1}{t+1}  \right\rceil$. Now $\left \lceil \frac{n-t+1}{t+1}  \right\rceil=1$, because $n=\hgt(\Gamma)-1$ and $ t \geq \left\lceil{\frac{\hgt(\Gamma)+1}{2}}\right\rceil.$ Also, $\sum\limits_{\ell_{\Gamma}(x)=\hgt(\Gamma) -t}^{\hgt(\Gamma)-2}\deg^{+}_{\Gamma}(x)=\sum\limits_{\ell_{\Gamma}(x)=\hgt(\Gamma) -t}^{\hgt(\Gamma)-2}1=t-1$. Thus the result holds in this case. Assume that $\hgt(\Gamma)>1$ and $n(\Gamma)>1$. Let us now suppose that the result holds for all rooted trees  $\Gamma''$ with $\hgt(\Gamma'')+n(\Gamma'')<\hgt(\Gamma)+n(\Gamma).$ If $t=\hgt(\Gamma)$ or $t=\hgt(\Gamma)+1$, using Theorem \ref{th1} and Theorem \ref{thm2}, we get the desired result. Thus we assume that $t \leq \hgt(\Gamma)-1.$ Let $\Gamma'=\Gamma \setminus \{v:\ell_{\Gamma}(v)=\hgt(\Gamma)\}.$ Then $t \leq \hgt(\Gamma')=\hgt(\Gamma)-1.$ Note that $\displaystyle \sum\limits_{\ell_{\Gamma}(x)=a}\deg_{\Gamma}^{+}(x)\leq \sum\limits_{\ell_{\Gamma}(x)=a+1}\deg_{\Gamma}^{+}(x)$ for all $a=0,1,\ldots,\hgt(\Gamma) -2$, because $\Gamma$ is a perfect rooted tree. Therefore, $n(\Gamma ') =\displaystyle \sum\limits_{\ell_{\Gamma}(x)=\hgt(\Gamma)-2}\deg_{\Gamma}^{+}(x)\leq \sum\limits_{\ell_{\Gamma}(x)=\hgt(\Gamma)-1}\deg_{\Gamma}^{+}(x)=n(\Gamma).$ Using induction, we have \begin{align*}
    \reg\left(\frac{R}{I_t(\Gamma')}\right) & \leq  \sum\limits_{\ell_{\Gamma'}(x)=\hgt(\Gamma')-t}^{\hgt(\Gamma')-2}\deg_{\Gamma'}^{+}(x)=\sum\limits_{\ell_{\Gamma}(x)=\hgt(\Gamma)-t-1}^{\hgt(\Gamma)-3}\deg_{\Gamma}^{+}(x) \leq \sum\limits_{\ell_{\Gamma}(x)=\hgt(\Gamma)-t}^{\hgt(\Gamma)-2}\deg_{\Gamma}^{+}(x).
\end{align*} 
Now in view of Theorem \ref{lem2}, it is enough to show that $\alpha(\Gamma) \leq \sum\limits_{\ell_{\Gamma}(x)=\hgt(\Gamma)-t}^{\hgt(\Gamma)-2}\deg_{\Gamma}^{+}(x). $
Let $z$ be any leaf of $\Gamma$. Following Notation \ref{nota},  $\Delta_j^{\Gamma}(z)=\emptyset$ or $\Delta_j^{\Gamma}(z)$ is a perfect rooted forest with $\hgt((\Delta_j^{\Gamma}(z)))=t-j-1$. Using Theorem \ref{th1}, we get 
\begin{eqnarray*}
& &\sum\limits_{j=0}^{t-2} \reg\left(\frac{R}{I_{t-j}(\Delta_j^{\Gamma}(z))} \right)+(t-1)\\&=& \sum\limits_{j=0}^{t-2} \left(\deg_{\Gamma_j(z)}^{+}(x_j(z))-1+\sum\limits_{\ell_{\Delta_j^{\Gamma}(z)}(x)=0}^{t-j-3}\deg^+_{\Delta_j^{\Gamma}(z)}(x)\right)+ (t-1)\\ 
&=& \sum\limits_{j=0}^{t-2} \left(\deg_{\Gamma_j(z)}^{+}(x_j(z))+\sum\limits_{\ell_{\Delta_j^{\Gamma}(z)}(x)=0}^{t-j-3}\deg^+_{\Delta_j^{\Gamma}(z)}(x) \right)\\ 
&=& \sum\limits_{j=0}^{t-2} \left(\deg_{\Gamma_j(z)}^{+}(x_j(z))+\sum\limits_{\ell_{\Gamma_j(z) \setminus \Gamma_{j+1}(z)}(x)=1}^{\hgt(\Gamma_j(z))-2}\deg^+_{{\Gamma_j(z) \setminus \Gamma_{j+1}(z)}}(x) \right)= \sum\limits_{\ell_{\Gamma_0(z)}(x)=0}^{\hgt(\Gamma_0(z))-2}\deg_{\Gamma_0(z)}^{+}(x).
\end{eqnarray*} Now, consider the following two cases:\\
{\bf Case 1:} Suppose that $\hgt( \Gamma(z)) < \hgt(\Gamma)$. In this  case, $\hgt( \Gamma(z))= \hgt(\Gamma)-t-1<t-1$, because $t \geq \frac{\hgt(\Gamma)+1}{2}.$ Thus $I_t\left( \Gamma(z) \right)=(0)$ which implies that \begin{align*}
   \reg\left( \frac{R}{I_t(\Gamma(z))}\right)+\sum\limits_{j=0}^{t-1} \reg\left(\frac{R}{I_{t-j}(\Delta_j^{\Gamma}(z))}\right)+(t-1)& =\sum\limits_{\ell_{\Gamma_0(z)}(x)=0}^{\hgt(\Gamma_0(z))-2}\deg_{\Gamma_0(z)}^{+}(x) =  \sum\limits_{\ell_{\Gamma}(x)=\hgt(\Gamma)-t}^{\hgt(\Gamma)-2}\deg_{\Gamma}^{+}(x),  
\end{align*} 
because $\ell_{\Gamma_0(z)}(x)=\ell_{\Gamma}(x)-\hgt(\Gamma)+t$ and $t=\hgt(\Gamma_0(z))$.  \\
{\bf Case II:} Suppose that $\hgt( \Gamma(z)) = \hgt(\Gamma)$. By Remark \ref{clean1}, $C(\Gamma(z))$ is a perfect rooted tree with $\hgt(C(\Gamma(z))) =\hgt(\Gamma(z))$ and by Remark \ref{clean}, $I_t(\Gamma(z))=I_t(C(\Gamma(z)))$. Since $n(C(\Gamma(z)))<n(\Gamma)$, by induction, we have 
\begin{align*}
\reg\left(\frac{R}{I_{t}(\Gamma(z))} \right)\leq \sum\limits_{\ell_{C(\Gamma(z))}(x)=\hgt(C(\Gamma(z)))-t}^{\hgt(C(\Gamma(z)))-2}\deg_{C(\Gamma(z))}^{+}(x)=\sum\limits_{\ell_{\Gamma(z)}(x)=\hgt(\Gamma(z))-t}^{\hgt(\Gamma(z))-2}\deg_{\Gamma(z)}^{+}(x).
\end{align*} Therefore, 
\begin{align*}
   & \reg\left(\frac{R}{I_{t}(\Gamma(z))} \right)+\sum\limits_{j=0}^{t-2} \reg\left(\frac{R}{I_{t-j}(\Delta_j^{\Gamma}(z))} \right)+(t-1) \\& \leq \sum\limits_{\ell_{\Gamma(z)}(x)=\hgt(\Gamma(z))-t}^{\hgt(\Gamma(z))-2}\deg_{\Gamma(z)}^{+}(x) + \sum\limits_{\ell_{\Gamma_0(z)}(x)=0}^{\hgt(\Gamma_0(z))-2}\deg_{\Gamma_0(z)}^{+}(x) =  \sum\limits_{\ell_{\Gamma}(x)=\hgt(\Gamma)-t}^{\hgt(\Gamma)-2}\deg_{\Gamma}^{+}(x),
\end{align*} because $\ell_{\Gamma(z)}(x) = \ell_{\Gamma}(x)$,  $\ell_{\Gamma_0(z)}(x)=\ell_{\Gamma}(x)-\hgt(\Gamma)+t$ and $t=\hgt(\Gamma_0(z))$. Thus, it follows from  both the cases that  $\alpha(\Gamma) \leq \sum\limits_{\ell_{\Gamma}(x)=\hgt(\Gamma)-t}^{\hgt(\Gamma)-2}\deg_{\Gamma}^{+}(x)$, and hence, the claim follows. Hence the assertion follows.
\end{proof}

   As a  consequence, we obtain an upper bound for the regularity of $t$-path ideal of a rooted tree when $t\ge \left\lceil{\frac{\hgt(\Gamma)+1}{2}}\right\rceil.$ Authors in \cite{{BRT}} gave an upper bound for the regularity of $t$-path ideals of rooted trees and proved that 
$$\reg\left(\frac{R}{I_t(\Gamma)}\right) \leq  (t-1)\left(l_t(\Gamma)+p_t(\Gamma)\right),$$ where $l_t(\Gamma )$ is the number of leaves in $\Gamma$ whose level
is at least $t-1$ and $p_t(\Gamma)$ is the maximal number of pairwise disjoint paths of length $t$ in $\Gamma$. In the following, for $\left\lceil{\frac{\hgt(\Gamma)+1}{2}}\right\rceil \leq t \leq \hgt(\Gamma) +1$, we obtain an improved upper bound for the regularity of $t$-path ideals of rooted trees which depends on the number of leaves whose level lies between $t-1$ and $\hgt(\Gamma)-2$ and the outdegree of the vertices whose level lies between $\hgt(\Gamma)-t$ and $\hgt(\Gamma)-2$. In case when the number of leaves at $\hgt(\Gamma)$ level or $\hgt(\Gamma)-1$ level is sufficiently large, the upper bound obtained in \cite{BRT} is comparatively huge than the bound given below.

\begin{corollary}\label{cormain}
Let $(\Gamma,x_0)$ be a rooted tree with $\hgt(\Gamma)\geq 1$ and let $t$ be a positive integer such that $\left\lceil{\frac{\hgt(\Gamma)+1}{2}}\right\rceil \leq t \leq \hgt(\Gamma) +1$. Then 
$$\reg\left(\frac{R}{I_t(\Gamma)}\right) \leq  \sum\limits_{\ell_{\Gamma}(x)=\hgt(\Gamma) -t}^{\hgt(\Gamma)-2}\deg^{+}_{\Gamma}(x)+\sum\limits_{i=t-1}^{\hgt(\Gamma)-2}(\hgt(\Gamma)-i-1)l_{i}(\Gamma),$$ where $l_{i}(\Gamma)$ denotes the number of leaves in $\Gamma$ at the level $i$.
\end{corollary}
\begin{proof}
By Remark \ref{clean}, $I_t(\Gamma)=I_t(C(\Gamma))$. Let $z$ be any leaf in $C(\Gamma)$ at level $i$. Then, $i=\ell_{\Gamma}(z) \geq t-1.$  Now, attach a rooted path $(P(z),z_i):z_{i},z_{i+1},\ldots,z_{\hgt(\Gamma)}$ of height $\hgt(\Gamma)-i$ in $C(\Gamma)$ at $z$. We repeat this process for all leaves in $C(\Gamma)$ to obtain a perfect rooted tree $(\Gamma^{'},x_0)$ of height $\hgt(\Gamma)$. By Theorem \ref{maint}, we have $$\reg\left(\frac{R}{I_t(\Gamma^{'})}\right) = \sum\limits_{\ell_{\Gamma'}(x)=\hgt(\Gamma') -t}^{\hgt(\Gamma')-2}\deg^{+}_{\Gamma'}(x)= \sum\limits_{\ell_{\Gamma}(x)=\hgt(\Gamma) -t}^{\hgt(\Gamma)-2}\deg^{+}_{\Gamma}(x)+\sum\limits_{i=t-1}^{\hgt(\Gamma)-2}(\hgt(\Gamma)-i-1)l_i(\Gamma).$$ Since $(\Gamma,x_0)$ is an induced subtree of $(\Gamma^{'},x_0)$, the rest follows from  Lemma \ref{lemm1}.
\end{proof}

\section{Regularity of powers of facet ideals of simplicial trees}\label{section5}
In this section, we provide a procedure to calculate the regularity of powers of facet ideals of simplicial trees. We
compute the regularity of powers of $t$-path ideals of some rooted trees. It is known from \cite{he_thesis} that  the simplicial complex whose facets are paths of length $t$ of a rooted forest is a simplicial forest.

\begin{notation}\label{not-1}
Let $\Delta$ be a simplicial forest on the vertex set $[n]$. Let $F_1,\ldots,F_r$ be the facets of $\Delta$.  Without loss of generality, assume that $F_1,\ldots,F_r$ is a good leaf ordering on the facets of $\Delta.$ For $1 \leq i \leq r-1,$ we set $\Delta_i=\langle F_1,\ldots,F_i  \rangle$ and $J_i=\langle m_{i+1},\ldots,m_r  \rangle.$
\end{notation}
In the following lemma, we compute certain colon ideals which we use to prove the main result of this section.
\begin{lemma}\label{lem1}
Let $\Delta$ be a simplicial forest as in  Notation \ref{not-1}. Then, for all $s \geq 1$, \begin{enumerate}
\item $I(\Delta)^{s+1} : m_r = I(\Delta)^s.$

\item $\left( I(\Delta_i)^{s+1}+J_i \right):m_i= I(\Delta_i)^s+(J_i:m_i)$ for all $1 \leq i \leq r-1$.

\item $\left( I(\Delta_1)^{s+1}+J_1 \right)+(m_1)= I(\Delta)$.

\item $\left( I(\Delta_i)^{s+1}+J_i \right)+(m_i)= I(\Delta_{i-1})^{s+1}+J_{i-1}$ for all $2 \leq i \leq r-1$.
\end{enumerate}
\end{lemma}
\begin{proof} 
\begin{enumerate}[1)]
\item Since $F_r$ is a good leaf of $\Delta$, it follows from the proof of \cite[Theorem 5.1]{GHJMNN19} that $I(\Delta)^{s+1} : m_r = I(\Delta)^s.$
\item  Fix $ 1 \leq i \leq r-1$. Note that $\left( I(\Delta_i)^{s+1}+J_i \right):m_i= I(\Delta_i)^{s+1}:m_i+(J_i:m_i)$. Since $F_1,\ldots, F_r$ is a good leaf order on the facets of $\Delta$, $F_i$ is a good leaf of the subcomplex with facets $F_1,\ldots,F_i.$ Therefore, it follows from  the proof of \cite[Theorem 5.1]{GHJMNN19} that $I(\Delta_i)^{s+1} : m_i = I(\Delta_i)^s.$ Thus, $\left( I(\Delta_i)^{s+1}+J_i \right):m_i= I(\Delta_i)^s+(J_i:m_i)$ for all $1 \leq i \leq r-1$.

\item The facts that $J_1+(m_1)=I(\Delta)$ and $I(\Delta_1)=(m_1)$ give us the desired result.

\item Fix $2 \leq i \leq r-1$. Consider
\begin{eqnarray*}
\left( I(\Delta_i)^{s+1}+J_i \right)+(m_i)&=& \left(\left( I(\Delta_{i-1})+m_i\right)^{s+1}+J_i \right)+(m_i)  \\
&=& \sum\limits_{j=0}^{s+1}m_i^jI(\Delta_{i-1})^{s+1-j}+J_i+(m_i) \\
&=& I(\Delta_{i-1})^{s+1}+J_i+(m_i) \\
&=& I(\Delta_{i-1})^{s+1}+J_{i-1}.
\end{eqnarray*} Hence, the assertion follows.
\end{enumerate}
\end{proof}
\begin{theorem} \label{simplicial}
Let $\Delta$ be a simplicial forest on the  vertex set $[n]$ as in Notation \ref{not-1}, and let $d_i=\deg(m_i)$ for $1 \leq i \leq r$. Then for all $s \geq 1$,
\begin{align*}
&\reg{\left(\frac{R}{I(\Delta)^{s+1}}\right)}\leq\\&  \max \left\{d_r+\reg{\left(\frac{R}{I(\Delta)^{s}}\right)},\max\limits_{1 \leq i \leq r-1} \left\{ d_i+ \reg{\left(\frac{R}{I(\Delta_i)^{s}+(J_i:m_i)}\right)}\right\}, \reg\left(\frac{R}{I(\Delta)}\right) \right\}.
\end{align*}
\end{theorem}
\begin{proof} Using Lemma \ref{lem1}, we get the following short exact sequences: 
\[
0 \rightarrow \frac{R}{I(\Delta)^{s}}(-d_r) \xrightarrow{\cdot m_r} \frac{R}{I(\Delta)^{s+1}} \rightarrow   \frac{R}{I(\Delta_{r-1})^{s+1}+J_{r-1}} \rightarrow 0,
\] for $2 \leq i \leq r-1$
\[
0 \rightarrow  \frac{R}{I(\Delta_{i})^{s}+(J_{i}:m_{i})}(-d_{i}) \xrightarrow{\cdot m_i} \frac{R}{I(\Delta_{i})^{s+1}+J_{i}} \rightarrow \frac{R}{I(\Delta_{i-1})^{s+1}+J_{i-1}} \rightarrow 0,
\] and \[
0 \rightarrow  \frac{R}{ I(\Delta_1)^{s}+(J_1 :m_1)}(-d_{1}) \xrightarrow{\cdot m_1} \frac{R}{I(\Delta_{1})^{s+1}+J_1} \rightarrow \frac{R}{I(\Delta)} \rightarrow 0.
\]
Now using \cite[Lemma 1.2]{HTT}, we get 
\begin{equation}
\reg{\left(\frac{R}{I(\Delta)^{s+1}}\right)} \leq \max \left\{d_r+\reg{\left(\frac{R}{I(\Delta)^{s}}\right)}, \reg{\left(\frac{R}{I(\Delta_{r-1})^{s+1}+J_{r-1}}\right)} \right\},
\end{equation} for $2 \leq i \leq r-1$
\begin{equation}
\reg{\left(\frac{R}{I(\Delta_{i})^{s+1}+J_{i}}\right)} \leq \max \left\{d_{i}+\reg{\left(\frac{R}{I(\Delta_{i})^{s}+(J_{i}:m_{i})}\right)}, \reg{\left(\frac{R}{I(\Delta_{i-1})^{s+1}+J_{i-1}} \right)} \right\}
\end{equation} and \begin{equation}
\reg{\left(\frac{R}{I(\Delta_{1})^{s+1}+J_{1}}\right)} \leq \max \left\{d_{1}+\reg{\left(\frac{R}{I(\Delta_{1})^{s}+(J_{1}:m_{1})}\right)},\reg\left(\frac{R}{I(\Delta)}\right) \right\}.
\end{equation}
Combining Equations $(1),(2)$ and $(3)$, we obtain the desired result.
\end{proof}

We use the above result to obtain the regularity of powers of some special class of simplicial trees.

\begin{definition}
A broom graph of height $h$ is a rooted tree $(\Gamma,x_0)$ that consists of a handle, which is a directed path $x_{0},x_{1},\ldots,x_{h}$ such that every other vertex (not on the handle) is a leaf vertex.
\end{definition}
Now we calculate the regularity of $\reg\left(\frac{R}{I_t(\Gamma)}\right)$ when $\Gamma$ is a broom graph.
\begin{theorem}\label{bg}
Let $(\Gamma,x_0)$ be a broom graph with $\hgt(\Gamma)=h \geq t-1$ with handle $x_0,x_1,\ldots,x_{h}$. Then $$\reg\left(\frac{R}{I_t(\Gamma)}\right)=(t-1)\left\lceil \frac{h-t+2}{t+1}\right\rceil.$$
\end{theorem}
\begin{proof}
First we prove that $\reg\left(\frac{R}{I_t(\Gamma)}\right) \leq (t-1)\left\lceil \frac{h-t+2}{t+1}\right\rceil.$ We make induction on the number of vertices of $\Gamma$. Let $n=|V(\Gamma)|$.  If $n=1$, then $t=1$, and hence, the result is obvious. Assume that $n>1$ and the result holds for all broom graphs with the number of vertices less than $n$. Let $z$ be a leaf of $\Gamma$ such that $\ell_{\Gamma}(z) =\hgt(\Gamma).$ Following  Notation \ref{nota}, we  get $ \Delta_j^{\Gamma}(z) $'s are empty for each $0 \leq j \leq t-1$. Therefore, by Lemma \ref{lemma}, \begin{equation}\label{equation}
    \reg\left(\frac{R}{I_t(\Gamma)}  \right) = \max \left\{\reg\left(\frac{R}{I_t(\Gamma \setminus \{z\})}\right),\reg\left(\frac{R}{I_t(\Gamma(z))}\right)  +t-1\right\}. 
\end{equation}
Using induction, we get  \begin{align*}
    \reg\left(\frac{R}{I_t(\Gamma \setminus \{z\})}\right) & \leq
    \begin{cases}
     (t-1)\left\lceil \frac{h-t+2 }{t+1}\right\rceil & \text{if $\hgt(\Gamma \setminus \{z\})=h$}\\
  (t-1)\left\lceil \frac{h-t+1 }{t+1}\right\rceil  & \text{if $\hgt(\Gamma \setminus \{z\})=h-1$.}
    \end{cases} \\& \leq (t-1)\left\lceil \frac{h-t+2 }{t+1}\right\rceil.
\end{align*}  Thus in the view of Equation (\ref{equation}), it is enough to show that $\reg\left(\frac{R}{I_t(\Gamma(z))}\right)  +(t-1) \leq (t-1)\left\lceil \frac{h-t+2 }{t+1}\right\rceil$.
Now, suppose that $t=h$ or $h+1$. Then $\Gamma(z)=\emptyset$, and therefore,  $$\reg\left(\frac{R}{I_t(\Gamma(z))}\right)  +(t-1)=(t-1)=(t-1)\left\lceil \frac{h-t+2 }{t+1}\right\rceil.$$ 
Next, suppose that $t < h$. In this case, $\Gamma(z)$ is a nonempty broom graph with $\hgt(\Gamma(z))=h-t-1$. Again, using induction, we get
\begin{eqnarray*}
\reg\left(\frac{R}{I_t(\Gamma(z))}\right)  +(t-1)& \leq & (t-1)\left\lceil \frac{h-2t+1 }{t+1}\right\rceil +(t-1)\\
&=& (t-1)\left\lceil \frac{h-2t+1 }{t+1}+1\right\rceil \\
&=& (t-1)\left\lceil \frac{h-t+2 }{t+1}\right\rceil.
\end{eqnarray*} Thus, the regularity upper bound follows. 

Next, let $P_h$ be the induced rooted path of $\Gamma$  on the vertex set $\{x_0,x_1,\ldots,x_h\}$. By Lemma \ref{lemm1},  $\reg\left(\frac{R}{I_t(P_{h})}\right) \leq \reg\left(\frac{R}{I_t(\Gamma)}\right).$ Hence, the assertion follows from \cite[Corollary 5.4]{BRT}.
\end{proof}

Now, we compute the regularity of powers of $t$-path ideal of broom graphs. For that purpose, we set the following notation.

\begin{notation}\label{not-2}
Let $\Gamma$ be a broom graph of $\hgt(\Gamma)=h$ and $t \geq 2$.  For $0 \leq i < h$, let $l_i$ denote the number of leaves in $\Gamma$ at the level $i$ and $l_h$ denote the number of leaves at the level $h$ minus one. Assume that $l_i = 0 $ if $i<t-1$. Then, for $0 \leq i \leq h$, the number of vertices in $\Gamma$ at level $i$ is $l_i+1$. For $0 \leq i \leq h$, let $V_i(\Gamma)=\{ x_{(i,0)},x_{(i,1)},\ldots, x_{(i,l_i)}\}$  be the set of  vertices of $\Gamma$ that belongs to level $i$. Assume, without loss of generality, that $x_{(0,0)},\ldots,x_{(h,0)}$ are the vertices of the handle. Now, for $0 \leq i \leq h-t+1$ and $0 \leq j \leq l_i$, set $F_{(i,j)} =\{x_{(i,0)},x_{(i+1,0)},\ldots,x_{(i+t-2,0)}, x_{(i+t-1,j)}\}$ and $m_{(i,j)} =\prod_{x\in F_{(i,j)}} x$. Let $\Delta$  denote the simplicial complex whose facets are $\{F_{(i,j)} \; : \; 0 \leq i \leq h-t+1, 0 \leq j \leq l_i\}$. Then, the facet ideal $I(\Delta)$ of $\Delta$ is the $t$-path ideal of $\Gamma.$ Our aim is to find a good leaf ordering on the facets of $\Delta$. Let $A= \{ (i,j) \; : \; 0 \leq i \leq h-t+1, 0 \leq j \leq l_i\}.$ For $(i,j), (k,l) \in A$, we declare $(i,j) <  (k,l)$ if either $i< k$ or if $i=k$, then $j> l.$ We claim that the ordering on the facets of $\Delta$ induced by the order on $A$ is a good leaf ordering.  For $(i,j) \in A$, let $\Delta_{(i,j)}$ denote the simplicial complex whose facets are $\left\{ F_{(k,l)} \; : \; (k,l)< (i,j) \right\}\cup \{F_{(i,j)}\}.$ Then, by \cite{HHTZ}, it is enough to prove that $F_{(i,j)}$ is a good leaf of $\Delta_{(i,j)}.$ Let $(k,l) \in A$ be such that $(k,l) < (i,j)$. If $k=i$, then $j < l$ and therefore, $F_{(i,j)} \cap F_{(i,l)} =\{x_{(i,0)},x_{(i+1,0)},\ldots,x_{(i+t-2,0)}\}$. Suppose that $k \neq i$, then $k <i$. Now, we have following cases: when $l \neq 0$, then $F_{(i,j)} \cap F_{(k,l)} =\emptyset$ if $k \leq i-t+1$, otherwise $F_{(i,j)} \cap F_{(k,l)} =\{x_{(i,0)},x_{(i+1,0)},\ldots,x_{(k+t-2,0)}\} $. When $l=0$, then  $F_{(i,j)} \cap F_{(k,l)} =\emptyset$ if $k < i-t+1$, otherwise   $F_{(i,j)} \cap F_{(k,l)} =\{x_{(i,0)},x_{(i+1,0)},\ldots,x_{(k+t-1,0)}\} $. Consequently, the collection $ \{ F_{(i,j)} \cap F_{(k,l)} \; : \; (k,l) < (i,j) \}$ is a totally ordered set with respect to inclusion, and hence, $F_{(i,j)}$ is a good leaf of $\Delta_{(i,j)}.$  
\end{notation}

In the next result, we use Notation \ref{not-2}  and Theorem \ref{simplicial} to calculate the regularity of $t$-path ideals of broom graphs.

\begin{theorem}\label{broommain}
Let $(\Gamma, x_0)$ be a broom graph of $\hgt(\Gamma)=h$ and $2 \leq t \leq h+1$. Then for all $s \geq 1,$ $$\reg{\left(\frac{R}{I_t(\Gamma)^s}\right)} = t(s-1)+\reg\left(\frac{R}{I_t(\Gamma)}\right).$$ 
\end{theorem}
\begin{proof}
First we claim that $\reg{\left(\frac{R}{I_t(\Gamma)^s}\right)}\leq t(s-1)+\reg\left(\frac{R}{I_t(\Gamma)}\right).$ 
We proceed by induction on $s$. If $s=1$, then the result trivially holds. Assume that the result holds for $s$, i.e., for any broom graph $\Gamma'$ with $\hgt(\Gamma')=h'$ and $2 \leq t \leq h'+1,$
\[
\reg{\left(\frac{R}{I_t(\Gamma')^s}\right)}\leq t(s-1)+\reg{\left(\frac{R}{I_t(\Gamma')}\right)}.
\]
Using Notation \ref{not-2}, we know that $I_t(\Gamma) =I(\Delta)$, where $\Delta$ is a simplicial tree whose facets are $\{F_{(i,j)} \; : \; (i,j) \in A\}$. Also, the ordering on the facets of $\Delta$ induced by the total order on $A$ is a good leaf ordering. We set $J_{(i,j)}=\left( m_{(k,l)} \; : \;  (i,j) < (k,l) \right),$ for $(i,j) \in A$. 
 By Theorem \ref{simplicial}, we get 
\begin{equation}\label{broom}
\reg{\left(\frac{R}{I(\Delta)^{s+1}}\right)} \leq \max \left\{t+\reg{\left(\frac{R}{I(\Delta)^{s}}\right)}, \alpha, \reg\left(\frac{R}{I(\Delta)}\right) \right\},
\end{equation} where $$\alpha =\max\limits_{(i,j) \in A\setminus \{(h-t+1,0)\}} \left\{ t+ \reg{\left(\frac{R}{I(\Delta_{(i,j)})^{s}+(J_{(i,j)}:m_{(i,j)})}\right)}\right\}.$$
Note that $I(\Delta_{(i,j)})=\left(m_{(k,l)} \; : \; (k,l) \le (i,j) \right)$ for all $(i,j) \in A$. It is easy to see that $$J_{(i,j)}:m_{(i,j)} = ( x_{(i+t-1,l)}\;:\; l<j)+ I_t \left( \Gamma_{ {(i+t,0)}} \right)$$ if $ j \neq 0$ and $$J_{(i,0)}:m_{(i,0)} = ( x_{(i+t,l)}\;:\; l\leq l_{i+t})+ I_t \left( \Gamma_{ {(i+t+1,0)}} \right),$$ where $\Gamma_{(i,j)}$ denote the induced rooted subtree of  $\Gamma$ rooted at $x_{(i,j)}.$ 
Consequently, for all $(i,j) \in A$, $I(\Delta_{(i,j)})$ and $J_{(i,j)}:m_{(i,j)}$ are monomial ideals generated in a disjoint set of variables, and therefore, \begin{align*}
  \reg\left(\frac{R}{I(\Delta_{(i,j)})^s+(J_{(i,j)}:m_{(i,j)})}\right)&=\reg{\left(\frac{R}{I(\Delta_{(i,j)})^s}\right)}+\reg{\left(\frac{R}{J_{(i,j)}:m_{(i,j)}}\right)} . 
\end{align*} Let $\Gamma_{ \leq (i,j)}$ denote the induced broom graph of $\Gamma$ such that $I(\Delta_{(i,j)})=I_t(\Gamma_{\leq (i,j)}).$ Therefore, by induction on $s$, we get $\reg{\left(\frac{R}{I(\Delta_{(i,j)})^s}\right)}\leq t(s-1)+\reg{\left(\frac{R}{I_t(\Gamma_{\leq (i,j)})}\right)}$. Note that $\Gamma_{\leq (i,j)} \sqcup \Gamma_{(i+t,0)}$ and $\Gamma_{\leq (i,0)} \sqcup \Gamma_{(i+t+1,0)}$ are disjoint union of induced broom subgraphs of $\Gamma.$   Thus, if $j \neq 0$
\begin{align*}
  \reg\left(\frac{R}{I(\Delta_{(i,j)})^s+(J_{(i,j)}:m_{(i,j)})}\right)&=\reg{\left(\frac{R}{I(\Delta_{(i,j)})^s}\right)}+\reg{\left(\frac{R}{J_{(i,j)}:m_{(i,j)}}\right)} \\ & \leq   t(s-1)+\reg{\left(\frac{R}{I_t(\Gamma_{\leq (i,j)})}\right)}+\reg{\left(\frac{R}{I_t(\Gamma_{(i+t,0)})}\right)}  \\& =t(s-1)+\reg{\left(\frac{R}{I_t(\Gamma_{\leq (i,j)}\sqcup \Gamma_{(i+t,0)})}\right)} \\& \leq  t(s-1)+\reg{\left(\frac{R}{I_t(\Gamma)}\right)},
\end{align*} where the last inequality follows from Lemma \ref{lemm1}. Similarly, 
\begin{align*}
  \reg\left(\frac{R}{I(\Delta_{(i,0)})^s+(J_{(i,0)}:m_{(i,0)})}\right)&=\reg{\left(\frac{R}{I(\Delta_{(i,0)})^s}\right)}+\reg{\left(\frac{R}{J_{(i,0)}:m_{(i,0)}}\right)} \\ & \leq   t(s-1)+\reg{\left(\frac{R}{I_t(\Gamma_{\leq (i,0)})}\right)}+\reg{\left(\frac{R}{I_t(\Gamma_{(i+t+1,0)})}\right)}  \\& =t(s-1)+\reg{\left(\frac{R}{I_t(\Gamma_{\leq (i,0)}\sqcup \Gamma_{(i+t+1,0)})}\right)} \\& \leq  t(s-1)+\reg{\left(\frac{R}{I_t(\Gamma)}\right)}.
\end{align*}
Therefore, \begin{align*}
\alpha =     \max\limits_{(i,j) \in A\setminus \{(h-t+1,0)\}} \left\{ t+ \reg{\left(\frac{R}{I(\Delta_{(i,j)})^{s}+(J_{(i,j)}:m_{(i,j)})}\right)}\right\} \leq ts+\reg{\left(\frac{R}{I_t(\Gamma)}\right)}. 
\end{align*} Also, by applying induction, we have $$\reg{\left(\frac{R}{I(\Delta)^{s}}\right)}  \leq   t(s-1)+\reg{\left(\frac{R}{I_t(\Gamma)}\right)}.$$ Now, using  Equation \eqref{broom}  the claim follows.

Let $\gamma$ be the induced rooted subforest of $\Gamma$ on the vertex set $\sqcup_{k=0}^{\left\lceil \frac{h-t+2}{t+1}\right\rceil-1}F_{(k(t+1),0)}.$ Once could observe that $I_t(\gamma)$ is complete intersection. Therefore, using similar arguments as used in \cite[Corollary 3.8]{BCDMS} and in \cite[Lemma 4.4]{BHT}, it follows that
\begin{eqnarray*}\label{eq2}
\reg{\left(\frac{R}{I_t(\Gamma)^{s}}\right)}\ge  \reg{\left(\frac{R}{I_t(\gamma)^{s}}\right)}  
&=& ts+(t-1)\left(\left\lceil \frac{h-t+2}{t+1}\right\rceil -1\right)-1 \\
&=& t(s-1)+\reg{\left(\frac{R}{I_t(\Gamma)}\right)}, 
\end{eqnarray*} where the last equality follows from Theorem \ref{bg}.  Hence, the assertion follows.
\end{proof}

We now characterize rooted trees such that some power of their $t$-path ideals has a linear resolution. 
\begin{theorem}
Let $(\Gamma,x_0)$ be a rooted tree with  $\hgt(\Gamma) \geq t-1.$ Then  the following are equivalent:
\begin{enumerate}[{$(1)$}]
    \item $C(\Gamma)$ is a broom graph of height at most $2t -1.$
    \item ${I_t(\Gamma)}$ has a linear resolution.
     \item ${I_t(\Gamma)^{s}}$ has linear quotients for all $s.$
    \item ${I_t(\Gamma)^{s}}$ has linear quotients for some $s.$ 
    \item ${I_t(\Gamma)^{s}}$ has a linear resolution for some $s.$
    \item ${I_t(\Gamma)^{s}}$ has linear first syzygies for some $s.$
\end{enumerate}
\end{theorem}
\begin{proof}
The equivalence of $(1)$ and $(2)$ follows from \cite[Theorem 4.5]{BRT}. The rest follows from Theorem  \ref{linear}. Hence, the assertion follows. 
\end{proof}

Next, we compute the regularity of powers of $(\hgt(\Gamma)+1)$-path ideals of perfect rooted trees.
\begin{theorem}\label{perfectr}
Let $(\Gamma,x_0)$ be a perfect rooted tree with $\hgt(\Gamma) \geq 1$, and let $t=\hgt(\Gamma)+1$.  Then,  for all $s \geq 1$, $$\reg{\left(\frac{R}{I_t(\Gamma)^s}\right)} = t(s-1)+1+\sum_{\ell_{\Gamma}(x)=0}^{\hgt(\Gamma)-2} \deg_{\Gamma}^+(x) =t(s-1)+\reg\left(\frac{R}{I_t(\Gamma)}\right).$$
\end{theorem}
\begin{proof}
We use induction on $\hgt(\Gamma).$ The case for $\hgt(\Gamma)=1$ is trivial. So, we assume that $\hgt(\Gamma)>1$. As we have seen in Theorem \ref{th1} that $I_t(\Gamma)=(x_0)I_{t-1}(\Gamma \setminus\{x_0\})$. Therefore, $I_t(\Gamma)^s=(x_0^s)I_{t-1}(\Gamma \setminus\{x_0\})^s.$ Let  $x_{i_1},\ldots,x_{i_k}$ be descents of $x_{0}$ in $\Gamma$. Then $\Gamma \setminus \{x_0\}$ is the disjoint union of perfect rooted trees, say, $(\Gamma_1,x_{i_1}),\ldots, (\Gamma_k,x_{i_k})$ each of height $\hgt(\Gamma)-1.$ Note that $t-1=\hgt(\Gamma)=\hgt(\Gamma_j)+1$ for each $j$. Thus, by induction, $$\reg\left(\frac{R}{I_{t-1}(\Gamma_j)^s}\right) = (t-1)(s-1)+ 1+ \sum_{\ell_{\Gamma_j}(x)=0}^{\hgt(\Gamma_j)-2} \deg_{\Gamma_j}^+(x)= (t-1)(s-1)+\reg\left(\frac{R}{I_{t-1}(\Gamma_j)}\right).$$  Now, applying \cite[Lemma 2.3]{HTT} and \cite[Theorem 1.1]{nguyen_vu} repeatedly, we get that \begin{eqnarray*}
\reg\left(\frac{R}{I_t(\Gamma)^s}\right)&=& s+(t-1)(s-1)+\sum\limits_{j=1}^k\reg\left(\frac{R}{I_{t-1}(\Gamma_j)}\right)\\&= &s+(t-1)(s-1)+\sum_{j=1}^k \left( 1+ \sum\limits_{\ell_{\Gamma_j}(x)=0}^{\hgt(\Gamma_j)-2} \deg_{\Gamma_j}^{+}(x) \right)\\&=& t(s-1)+1+k+\sum\limits_{j=1}^k\sum\limits_{\ell_{\Gamma_j}(x) = 0}^{\hgt({\Gamma_j})-2}\deg_{\Gamma_j}^+(x)\\&=& t(s-1)+1+\sum_{\ell_{\Gamma}(x)=0}^{\hgt(\Gamma)-2} \deg_{\Gamma}^+(x) =t(s-1)+\reg\left(\frac{R}{I_t(\Gamma)}\right), 
\end{eqnarray*} where the last equality follows from Theorem \ref{th1}. Hence, the assertion follows. 
\end{proof}

Theorem \ref{broommain}, Theorem \ref{perfectr}, and computational evidence lead us to pose the following conjecture.
\begin{conjecture}
Let $\Delta$ be a $d$-dimensional simplicial tree on the vertex set $[n].$ Then, $$\reg\left(\frac{R}{I(\Delta)^s}\right) \leq (d+1)(s-1)+\reg\left(\frac{R}{I(\Delta)}\right) \text{ for all } s \geq 1.$$
\end{conjecture}

\bibliographystyle{plain}
\bibliography{refs_reg}

\begin{thebibliography}{10}

\bibitem{ABS}
Ali Alilooee, Selvi~Kara Beyarslan, and S.~Selvaraja.
\newblock Regularity of powers of edge ideals of unicyclic graphs.
\newblock {\em Rocky Mountain J. Math.}, 49(3):699--728, 2019.

\bibitem{banerjee}
A.~Banerjee.
\newblock The regularity of powers of edge ideals.
\newblock {\em J. Algebraic Combin.}, 41(2):303--321, 2015.

\bibitem{BBSHA}
Arindam Banerjee, Selvi~Kara Beyarslan, and Huy~T\`ai H\`a.
\newblock Regularity of powers of edge ideals: from local properties to global
  bounds.
\newblock {\em Algebraic Combinatorics}, 3(4):839--854, 2020.

\bibitem{BBH17}
Arindam Banerjee, Selvi~Kara Beyarslan, and H\`a Huy~T\`ai.
\newblock Regularity of edge ideals and their powers.
\newblock In {\em Advances in algebra}, volume 277 of {\em Springer Proc. Math.
  Stat.}, pages 17--52. Springer, Cham, 2019.

\bibitem{BCDMS}
Arindam Banerjee, Bidwan Chakraborty, Kanoy~Kumar Das, Mousumi Mandal, and
  S.~Selvaraja.
\newblock Regularity of powers of squarefree monomial ideals.
\newblock {\em J. Pure Appl. Algebra}, 226(2):Paper No. 106807, 12, 2022.

\bibitem{BHT}
Selvi Beyarslan, Huy~T{\`a}i H{\`a}, and Tr{\^a}n~Nam Trung.
\newblock Regularity of powers of forests and cycles.
\newblock {\em J. Algebraic Combin.}, 42(4):1077--1095, 2015.

\bibitem{RT}
Rachelle~R. Bouchat and Tricia~Muldoon Brown.
\newblock Multi-graded {B}etti numbers of path ideals of trees.
\newblock {\em J. Algebra Appl.}, 16(1):1750018, 20, 2017.

\bibitem{BRT}
Rachelle~R. Bouchat, Huy~T\`ai H\`a, and Augustine O'Keefe.
\newblock Path ideals of rooted trees and their graded {B}etti numbers.
\newblock {\em J. Combin. Theory Ser. A}, 118(8):2411--2425, 2011.

\bibitem{BC17}
Winfried Bruns and Aldo Conca.
\newblock A remark on regularity of powers and products of ideals.
\newblock {\em J. Pure Appl. Algebra}, 221(11):2861--2868, 2017.

\bibitem{BCV15}
Winfried Bruns, Aldo Conca, and Matteo Varbaro.
\newblock Maximal minors and linear powers.
\newblock {\em J. Reine Angew. Math.}, 702:41--53, 2015.

\bibitem{GHJMNN19}
Giulio Caviglia, Huy~T\`ai H\`a, J\"{u}rgen Herzog, Manoj Kummini, Naoki Terai,
  and Ngo~Viet Trung.
\newblock Depth and regularity modulo a principal ideal.
\newblock {\em J. Algebraic Combin.}, 49(1):1--20, 2019.

\bibitem{RYJN}
Yairon Cid-Ruiz, Sepehr Jafari, Navid Nemati, and Beatrice Picone.
\newblock Regularity of bicyclic graphs and their powers.
\newblock {\em J. Algebra Appl.}, 19(3):2050057, 38, 2020.

\bibitem{Conca2000}
Aldo Conca.
\newblock Hilbert function and resolution of the powers of the ideal of the
  rational normal curve.
\newblock volume 152, pages 65--74. 2000.
\newblock Commutative algebra, homological algebra and representation theory
  (Catania/Genoa/Rome, 1998).

\bibitem{FC}
E.~Connon and Sara Faridi.
\newblock A criterion for a monomial ideal to have a linear resolution in
  characteristic 2.
\newblock {\em Electron. J. Combin.}, 22(1):Paper 1.63, 15, 2015.

\bibitem{CHT}
S.~D. Cutkosky, J.~Herzog, and N.~V. Trung.
\newblock Asymptotic behaviour of the {C}astelnuovo-{M}umford regularity.
\newblock {\em Compositio Math.}, 118(3):243--261, 1999.

\bibitem{HHJ13}
Hailong Dao, Craig Huneke, and Jay Schweig.
\newblock Bounds on the regularity and projective dimension of ideals
  associated to graphs.
\newblock {\em J. Algebraic Combin.}, 38(1):37--55, 2013.

\bibitem{Er}
N.~Erey.
\newblock Powers of edge ideals with linear resolutions.
\newblock {\em Comm. Algebra}, 46(9):4007--4020, 2018.

\bibitem{Er2}
N.~Erey.
\newblock Powers of ideals associated to {$(C_4,2K_2)$}-free graphs.
\newblock {\em J. Pure Appl. Algebra}, 223(7):3071--3080, 2019.

\bibitem{SF}
Sara Faridi.
\newblock The facet ideal of a simplicial complex.
\newblock {\em Manuscripta Math.}, 109(2):159--174, 2002.

\bibitem{froberg}
Ralf Fr{\"o}berg.
\newblock On {S}tanley-{R}eisner rings.
\newblock In {\em Topics in algebra, {P}art 2 ({W}arsaw, 1988)}, volume~26 of
  {\em Banach Center Publ.}, pages 57--70. PWN, Warsaw, 1990.

\bibitem{Gu}
Yan Gu.
\newblock Regularity of {P}owers of {E}dge {I}deals of {S}ome {G}raphs.
\newblock {\em Acta Math. Vietnam.}, 42(3):445--454, 2017.

\bibitem{GVa05}
Elena Guardo and Adam Van~Tuyl.
\newblock Powers of complete intersections: graded {B}etti numbers and
  applications.
\newblock {\em Illinois J. Math.}, 49(1):265--279, 2005.

\bibitem{HTT}
Huy~T{\`a}i H{\`a}, Ngo~Viet Trung, and Tr{\^a}n~Nam Trung.
\newblock Depth and regularity of powers of sums of ideals.
\newblock {\em Math. Z.}, 282(3-4):819--838, 2016.

\bibitem{HaWood}
Huy~T\`ai H\`a and Russ Woodroofe.
\newblock Results on the regularity of square-free monomial ideals.
\newblock {\em Adv. in Appl. Math.}, 58:21--36, 2014.

\bibitem{he_thesis}
Jing He.
\newblock {\em The path ideal of a tree and its properties}.
\newblock PhD thesis, Lakehead University, 2007.

\bibitem{VanH}
Jing He and Adam Van~Tuyl.
\newblock Algebraic properties of the path ideal of a tree.
\newblock {\em Comm. Algebra}, 38(5):1725--1742, 2010.

\bibitem{Herzog'sBook}
J.~Herzog and T.~Hibi.
\newblock {\em Monomial ideals}, volume 260 of {\em Graduate Texts in
  Mathematics}.
\newblock Springer-Verlag London, Ltd., London, 2011.

\bibitem{HHZ}
J.~Herzog, T.~Hibi, and X.~Zheng.
\newblock Monomial ideals whose powers have a linear resolution.
\newblock {\em Math. Scand.}, 95(1):23--32, 2004.

\bibitem{HHTZ}
J\"{u}rgen Herzog, Takayuki Hibi, Ng\^{o}~Vi\^{e}t Trung, and Xinxian Zheng.
\newblock Standard graded vertex cover algebras, cycles and leaves.
\newblock {\em Trans. Amer. Math. Soc.}, 360(12):6231--6249, 2008.

\bibitem{HTY}
J\"{u}rgen Herzog and Yukihide Takayama.
\newblock Resolutions by mapping cones.
\newblock volume~4, pages 277--294. 2002.
\newblock The Roos Festschrift volume, 2.

\bibitem{hien2021}
Truong~Thi Hien and Tran~Nam Trung.
\newblock Regularity of symbolic powers of square-free monomial ideals, 2021.

\bibitem{JKS20}
A.~V. Jayanthan, Arvind Kumar, and Rajib Sarkar.
\newblock Regularity of powers of quadratic sequences with applications to
  binomial ideals.
\newblock {\em J. Algebra}, 564:98--118, 2020.

\bibitem{jayanthan}
A.~V. Jayanthan, N.~Narayanan, and S.~Selvaraja.
\newblock Regularity of powers of bipartite graphs.
\newblock {\em J. Algebraic Combin.}, 47(1):17--38, 2018.

\bibitem{JSS}
A.~V. Jayanthan and S.~Selvaraja.
\newblock Linear polynomials for the regularity of powers of edge ideals of
  very well-covered graphs.
\newblock {\em J. Commut. Algebra}, 13(1):89--101, 2021.

\bibitem{JS}
A.~V. Jayanthan and S.~Selvaraja.
\newblock Upper bounds for the regularity of powers of edge ideals of graphs.
\newblock {\em J. Algebra}, 574:184--205, 2021.

\bibitem{KMS}
Dariush Kiani and Sara Saeedi~Madani.
\newblock Betti numbers of path ideals of trees.
\newblock {\em Comm. Algebra}, 44(12):5376--5394, 2016.

\bibitem{vijay}
V.~Kodiyalam.
\newblock Asymptotic behaviour of {C}astelnuovo-{M}umford regularity.
\newblock {\em Proc. Amer. Math. Soc.}, 128(2):407--411, 2000.

\bibitem{KK23}
Ajay Kumar and Rajiv Kumar.
\newblock Regularity of powers of bipartite graphs.
\newblock {\em Bull. Aust. Math. Soc.}, 107(1):1--9, 2023.

\bibitem{KKS21}
Arvind Kumar, Rajiv Kumar, and Rajib Sarkar.
\newblock Certain algebraic invariants of edge ideals of join of graphs.
\newblock {\em J. Algebra Appl.}, 20(6):Paper No. 2150099, 12, 2021.

\bibitem{nguyen_vu}
Hop~D. Nguyen and Thanh Vu.
\newblock Powers of sums and their homological invariants.
\newblock {\em J. Pure Appl. Algebra}, 223(7):3081--3111, 2019.

\bibitem{Sturmfels}
B.~Sturmfels.
\newblock Four counterexamples in combinatorial algebraic geometry.
\newblock {\em J. Algebra}, 230(1):282--294, 2000.

\bibitem{XZ}
Xinxian Zheng.
\newblock {\em Homological properties of monomial ideals associated to
  quasi-trees and lattices}.
\newblock PhD thesis, University of Essen, 2004.

\bibitem{Zheng}
Xinxian Zheng.
\newblock Resolutions of facet ideals.
\newblock {\em Comm. Algebra}, 32(6):2301--2324, 2004.

\end{thebibliography}
\end{document}